\documentclass[12pt]{amsart}
\usepackage{mathptmx}
\usepackage[latin9]{inputenc}
\usepackage{color}
\usepackage{mathtools}
\usepackage{amstext}
\usepackage{amsthm}
\usepackage{amssymb}
\usepackage{graphicx}
\usepackage{shuffle}
\usepackage[unicode=true,pdfusetitle,
 bookmarks=true,bookmarksnumbered=false,bookmarksopen=false,
 breaklinks=false,pdfborder={0 0 0},pdfborderstyle={},backref=false,colorlinks=true]
 {hyperref}
\hypersetup{
 pdfborderstyle={},pdfborderstyle={},pdfborderstyle={}}

\makeatletter
\numberwithin{equation}{section}
\numberwithin{figure}{section}
\theoremstyle{plain}
\newtheorem{thm}{\protect\theoremname}[section]
\theoremstyle{definition}
\newtheorem{defn}[thm]{\protect\definitionname}
\theoremstyle{remark}
\newtheorem{rem}[thm]{\protect\remarkname}
\theoremstyle{plain}
\newtheorem{lem}[thm]{\protect\lemmaname}
\theoremstyle{definition}
\newtheorem{condition}{\protect\conditionname}
\theoremstyle{plain}
\newtheorem{prop}[thm]{\protect\propositionname}
\theoremstyle{plain}
\newtheorem{cor}[thm]{\protect\corollaryname}
\theoremstyle{definition}
\newtheorem{example}[thm]{\protect\examplename}

\@ifundefined{date}{}{\date{}}

\usepackage{amsfonts}
\usepackage{dsfont}

\numberwithin{equation}{section}

\allowdisplaybreaks[4]
\DeclareMathAlphabet{\mathcal}{OMS}{cmsy}{m}{n}

  \providecommand{\corollaryname}{Corollary}
  \providecommand{\lemmaname}{Lemma}
  \providecommand{\remarkname}{Remark}
\providecommand{\theoremname}{Theorem}
\providecommand{\definitionname}{Definition}
\providecommand{\examplename}{Example}
\providecommand{\propositionname}{Proposition}
\providecommand{\conditionname}{Condition}

  \providecommand{\corollaryname}{Corollary}
  \providecommand{\lemmaname}{Lemma}
  \providecommand{\remarkname}{Remark}
\providecommand{\theoremname}{Theorem}
\providecommand{\definitionname}{Definition}
\providecommand{\examplename}{Example}
\providecommand{\propositionname}{Proposition}

  \providecommand{\corollaryname}{Corollary}
  \providecommand{\definitionname}{Definition}
  \providecommand{\examplename}{Example}
  \providecommand{\lemmaname}{Lemma}
  \providecommand{\propositionname}{Proposition}
  \providecommand{\remarkname}{Remark}
\providecommand{\theoremname}{Theorem}

\makeatother

\providecommand{\conditionname}{Condition}
\providecommand{\corollaryname}{Corollary}
\providecommand{\definitionname}{Definition}
\providecommand{\examplename}{Example}
\providecommand{\lemmaname}{Lemma}
\providecommand{\propositionname}{Proposition}
\providecommand{\remarkname}{Remark}
\providecommand{\theoremname}{Theorem}

\begin{document}
\title{General Signature Kernels}
\author{Thomas Cass}\thanks{Imperial College London and the Alan Turing Institute, thomas.cass@imperial.ac.uk}
\author{Terry Lyons}\thanks{University of Oxford and the Alan Turing Institute, t.lyons@maths.ox.ac.uk}
\author{Xingcheng Xu}\thanks{Imperial College London, xingcheng.xu18@gmail.com} 

\date{%
\today
}
\thanks{The work of all three authors was supported by the EPSRC Programme Grant EP/S026347/1}
\begin{abstract}
{\normalsize{} Suppose that $\gamma$ and $\sigma$ are two continuous bounded variation paths which take values in a finite-dimensional inner product space $V$. The recent papers \cite{ko} and \cite{clsw} respectively introduced the truncated and the untruncated signature kernel of $\gamma$ and $\sigma$ and showed how these concepts can be used in classification and prediction tasks involving multivariate time series. In this paper we consider general signature kernels of the form 
\begin{equation}
K^{\gamma,\sigma}_{\phi}\left(  s,t\right)  =\left\langle S\left(  \gamma\right)
_{a,s},S\left(  \sigma\right)  _{a,t}\right\rangle_{\phi} :=\sum_{k=0}^{\infty
}\phi(k)\left\langle S\left(  \gamma\right)  _{a,s}^{k},S\left(  \sigma\right)
_{a,t}^{k}\right\rangle _{k}\label{orig sig ker}%
\end{equation}
where $\left\langle \cdot,\cdot\right\rangle _{k}$ is the Hilbert-Schmidt inner-product on $V^{\otimes k}$ and $\phi: \mathbb{N}\cup \{0\} \mapsto \mathbb{C}$. We show how $K^{\gamma,\sigma}_{\phi}$ can be interpreted in many examples as an average of PDE solutions and thus how it can estimated computationally using suitable quadrature formulae. We extend this analysis to derive closed-form formulae for expressions involving the expected (Stratonovich)
signature of Brownian motion. In doing so we articulate a novel connection between signature kernels and the hyperbolic
development, the latter of which has been a broadly useful tool in the analysis of the signature, see e.g. \cite{HL2010}, \cite{LX} and \cite{BG2019}. As an application we evaluate the use different general signature kernels as the basis for non-parametric goodness-of-fit tests to Wiener measure on path space.}{\normalsize\par}
\end{abstract}

\maketitle
\vspace{5mm}
 \hspace{8mm}\textbf{Keywords:} The signature, expected signatures, kernel methods, general signature kernels,

\hspace{8mm}Gaussian quadrature, hyperbolic development, contour integration

\section{Introduction}

Kernel methods are well-established tools in machine learning which are 
fundamental to support vector machine models for classification, nonlinear
regression and outlier detection involving small or moderate-sized data sets \cite{steinwart}, \cite{carmeli}, \cite{Sri}. Applications are manifold and include text classification \cite{lodhi}, protein classification \cite{leslie} as well as applications to biological sequences \cite{tsuda} and labelled graphs \cite{kashima}.
The essence of these methods is to achieve better separation between labelled data by embedding a low-dimensional feature space $X$ into a higher dimensional one $H$, which is commonly assumed to be a Hilbert space, by
means of a feature map $\psi:X\rightarrow H$. The associated kernel is a
function $K:X\times X\rightarrow%
\mathbb{R}
$ with the property that $\left\langle \psi\left(  x\right)  ,\psi\left(
y\right)  \right\rangle _{H}=K\left(  x,y\right)  $ for all $x$ and $y$ in
$X.$ If $K$ is known in closed form then the
inner-products of all extended features are obtainable from the evaluation of
$K$ at pairs of training instances in the original feature set. A typical classification problem can be formulated as convex constrained
optimisation problem for which the Lagrangian dual involves only the
inner-products of pairs of enhanced features in the set of training instances. Crucially, one does not need the vectors of the enhanced features themselves. This observation -- the
basis of the so-called kernel trick -- then allows one to enjoy the advantages of
working in a higher dimensional feature space without some of the concomitant drawbacks.

The selection of an effective kernel is challenging and somewhat
task-dependent. When the training data consist of sequential data such as time
series, these challenges are magnified. To address these and other
difficulties much recent progress has been made by
re-purposing the (path) signature transform from rough path theory, which has
decisive advantages in capturing complex interactions between multivariate
data streams. We recall that the signature of a continuous bounded
variation path $\gamma:\left[  a,b\right]  \rightarrow V$ is the formal tensor
series of iterated integrals
\begin{equation}
S\left(  \gamma\right)  _{a,b}=1+\sum_{k=1}^{\infty}S\left(  \gamma\right)
_{a,b}^{k}\in T\left(  \left(  V\right)  \right)  =\oplus_{k=0}^{\infty
}V^{\otimes k}\text{ with }S\left(  \gamma\right)  _{a,b}^{k}:=\int%
_{a<t_{1}<t_{2}<...<t_{k}<b}d\gamma_{t_{1}}\otimes...\otimes d\gamma_{t_{k}%
}.\label{sig def 1}%
\end{equation}
The soundness of this approach is underpinned by the fact that the map
$\gamma\mapsto S\left(  \gamma\right)  _{a,b}$ is one-to-one, up to an
equivalence relation on the space of paths \cite{HL2010}. The signature is invariant under reparameterisation, and
therefore by representing the path $\gamma$ by the tensor series
$S(\gamma)_{a,b}$ one removes an otherwise complicating infinite-dimensional
symmetry$.$ On the other hand, the signature captures the order of
events along $\gamma.$ The algebraic properties of the signature have been developed since the foundational work of Chen; it is now well understood that the signature transform describes the set of polynomials on unparameterised paths, in a sense that can be made meaningful. Analytically,
the signature of $\gamma$ characterises the class of responses (i.e.
solutions) of all smooth differential systems which have $\gamma$ as the input.

An important fact is the factorial decay rate of the terms in the series in
(\ref{sig def 1}).\ That is, given appropriately defined norms on the tensor
product spaces $V^{\otimes k}:$
\[
\left\vert \left\vert \int_{a<t_{1}<t_{2}<...<t_{k}<b}d\gamma_{t_{1}}%
\otimes...\otimes d\gamma_{t_{k}}\right\vert \right\vert _{V^{\otimes k}}%
\leq\frac{L\left(  \gamma\right)  ^{k}}{k!},
\]
where $L\left(  \gamma\right)  $ denotes the length of the path over $\left[
a,b\right]  .$ This allows one to define the (untruncated) signature kernel of
two paths $\gamma$ and $\sigma$ by
\begin{equation}
K^{\gamma,\sigma}\left(  s,t\right)  =\left\langle S\left(  \gamma\right)
_{a,s},S\left(  \sigma\right)  _{a,t}\right\rangle :=1+\sum_{k=1}^{\infty
}\left\langle S\left(  \gamma\right)  _{a,s}^{k},S\left(  \sigma\right)
_{a,t}^{k}\right\rangle _{k}\label{orig sig ker}%
\end{equation}
where $\left\langle \cdot,\cdot\right\rangle _{k}$ is the canonical
(Hilbert-Schmidt) inner-product on $V^{\otimes k}$ derived from a fixed
inner-product on $V$. In the recent paper \cite{clsw} it was shown that this untruncated
signature kernel has some advantages over it truncated counterpart \cite{ko} which, in some cases, lead to greater accuracy in classification
and regression tasks on benchmark data sets for multivariate time series. The explanation
for this turns on the key observation that $K$ is the unique solution of the
hyperbolic partial differential equation%
\begin{equation}
\frac{\partial^{2}K}{\partial s\partial t}\left(  s,t\right)  =K\left(
s,t\right)  \left\langle \gamma_{s}^{\prime},\sigma_{t}^{\prime}\right\rangle
\text{ with }K\left(  a,\cdot\right)  =K\left(  \cdot,a\right)  \equiv
1.\label{hpde}%
\end{equation}
The solution to which can be approximated using PDE solvers, thus allowing for the
efficient computation of the inner product in (\ref{orig sig ker}) and obviating the need to compute iterated integrals.

While the kernel (\ref{orig sig ker}) is useful, it is also in some respects
confining. One restriction it imposes is on the relative contributions made to
the sum (\ref{orig sig ker}) by the different inner-products $\left\langle
\cdot,\cdot\right\rangle _{k}$. It is easy to see for example by scaling
$\gamma$ by $\lambda=e^{\alpha}\in%
\mathbb{R}
$ to give $(\lambda\gamma)_{\cdot}=\lambda\gamma_{\cdot}$ we have
\[
K^{\lambda \gamma,\sigma}\left(  s,t\right)  =1+\sum_{k=1}^{\infty}e^{\alpha
k}\left\langle S\left(  \gamma\right)  _{a,s}^{k},S\left(  \sigma\right)
_{a,t}^{k}\right\rangle _{k},
\]
so that the signature kernel for the family of inner-products $\left\langle
\cdot,\cdot\right\rangle _{\alpha}=\sum_{k\geq0}e^{\alpha k}\left\langle
\cdot,\cdot\right\rangle _{k}$ can be obtained as above by solving the
appropriately rescaled version of the PDE (\ref{hpde}). The starting point for
this paper is to introduce methods that allow for the efficient computation of
general signature kernels with a different weighting. These will be derived
from bilinear forms on $T\left(  V\right)  $ of the type
\[
\left\langle \cdot,\cdot\right\rangle _{\phi}=\sum_{k=0}^{\infty}\phi\left(
k\right)  \left\langle \cdot,\cdot\right\rangle _{k},
\]
where $\phi:\mathbb{N\cup}\left\{  0\right\}  \rightarrow%
\mathbb{R}
$ (or, sometimes, $%
\mathbb{C}
$), so that $\left\langle \cdot,\cdot\right\rangle _{\phi}$ need not even define
an inner-product. One fundamental observation we take advantage of is
illustrated by the following argument: assume $\phi\left(  0\right)  =1, $ and
suppose that we can solve the Hamburger moment problem for the sequence
$\left\{  \phi\left(  k\right)  :k\in\mathbb{N\cup}\left\{  0\right\}
\right\}  $, i.e. we can find a probability measure $\mu$ on $%
\mathbb{R}
$ such that
\begin{equation}
\phi\left(  k\right)  =\int\lambda^{k}d\mu\left(  \lambda\right)  \text{ for
all }k\in\mathbb{N\cup}\left\{  0\right\}  .\label{hmp}%
\end{equation}
Then, under some conditions on $\mu,$ we will be able to justify the following
identity
\begin{equation}
\left\langle S\left(  \gamma\right)  _{a,s},S\left(  \sigma\right)
_{a,t}\right\rangle _{\phi}=\sum_{k=0}^{\infty}\int\lambda^{k}\left\langle
S\left(  \gamma\right)  _{a,s}^{k},S\left(  \sigma\right)  _{a,t}%
^{k}\right\rangle _{k}d\mu\left(  \lambda\right)  =\int K^{\lambda\gamma,\sigma}\left(  s,t\right)  d\mu\left(  \lambda\right)  .\label{key identity}%
\end{equation}
In this case, the computation of the $\phi-$signature kernel, i.e. the one
arising from $\left\langle \cdot,\cdot\right\rangle _{\phi},$ will amount to
integrating scaled solutions to the PDE (\ref{hpde}) in $\lambda$ with respect
to the measure $\mu.$ The practicability of this approach depends on two
aspects. Firstly, one needs to be able to solve the moment problem
(\ref{hmp}); there are well-known necessary and sufficient conditions but,
ideally, $\mu$ should be determined explicitly. Secondly, one needs to be able
to approximate accurately the integral on the right hand side of
(\ref{key identity}). In this respect one is helped by the form of the
function $\lambda\mapsto K^{\lambda\gamma,\sigma}\left(  s,t\right)  $
which is real analytic with a power series whose coefficients decay at rate
$\left(  n!\right)  ^{-2}$. Hence, in cases where $\mu$ has a density $w$
given in closed form, Gaussian quadrature provides an approximation of the
form
\[
\int K^{\lambda\gamma,\sigma}\left(  s,t\right)  d\mu\left(  \lambda
\right)  \approx\sum_{i=1}^{m}w_{i}K^{\lambda_{i}\gamma,\sigma}\left(
s,t\right)
\]
and equip us with well-described error bounds, see e.g. \cite{suli}. For these
examples, the $\phi-$signature kernel can be approximated at the expense of
$m$ implementations of a PDE solver.

The same principle outlined in the previous paragraph can appear in different
guises. For example, by solving the trigonometric moment problem
\[
\phi\left(  k\right)  =\int_{0}^{2\pi}e^{ik\theta}d\mu\left(  \theta\right)
\text{ for }k\in%
\mathbb{Z}%
\]
to find a measure $\mu$ on $\left[  0,2\pi\right]$, then an analogue of
(\ref{key identity}) can be obtained by integrating the complex-valued
function $\theta\mapsto$ $K^{\exp\left(  i\theta\right)\gamma,\sigma
}\left(  s,t\right)  $ with respect to $\mu.$ A similar observation applies to
a class of integral transforms having the form
\begin{equation}
\phi\left(  u\right)  =\int_{C}r\left(  u,z\right)  d\mu\left(  z\right)
;\text{ where }r\left(  u,z\right)  =g\left(  z\right)  ^{\alpha u}\text{ }\in%
\mathbb{C}
\text{ for some }\alpha\in%
\mathbb{R}
.\label{pairs}%
\end{equation}
This class includes the Fourier-, Laplace- and Mellin- Stieltjes transforms,
for which specific pairs $\left(  \phi,\mu\right)  $ are of course extensively
documented. We illustrate a range of examples that can be generated using this
idea in the main text.

Extensions of the same idea apply to expected signatures. It is by now well
known that, under some conditions, the expected signature of a stochastic
process characterises the law of that process \cite{CL-2016}. This motivates
the use of expected signatures as a measure of similarity of two laws on path
space, for example through the quantity%
\[
d_{\phi}\left(  \mu,\nu\right)  :=\left\vert \left\vert \mathbb{E}^{\mu
}\left[  S\left(  X\right)  \right]  -\mathbb{E}^{\nu}\left[  S\left(
X\right)  \right]  \right\vert \right\vert _{\phi},
\]
which is seen to be a maximum mean discrepancy (MMD) distance between $\mu$
and $\nu$; see \cite{Gretton-2012} and \cite{CO-2018}. We also have a measure
of alignment of the two expected signatures of $\mu$ and $\nu$ given by
\[
\cos\angle_{\phi}\left(  \mu,\nu\right)  :=\frac{\left\langle \mathbb{E}^{\mu
}\left[  S\left(  X\right)  \right]  ,\mathbb{E}^{\nu}\left[  S\left(
X\right)  \right]  \right\rangle _{\phi}}{\left\vert \left\vert \mathbb{E}%
^{\mu}\left[  S\left(  X\right)  \right]  \right\vert \right\vert _{\phi
}\left\vert \left\vert \mathbb{E}^{\nu}\left[  S\left(  X\right)  \right]
\right\vert \right\vert _{\phi}},
\]
which can be interpreted as an analogue of the Pearson correlation coefficient for measures on path space.
As an application we consider designing goodness-of -fit tests in which one wants to understand when an
observed empirical sample is drawn from a well-described baseline
distribution. A motivating example for this paper was that of the
detection of radio frequency interference (RFI) contamination in
radioastronomy. In this situation, electrical signals are collected from an
array of antennas \cite{Wilensky2019}. Under the null hypothesis of no RFI
contamination, the signals will reflect only the so-called thermal noise of
the receiving equipment. From this perspective, the most important reference distribution will that of white
noise or, in its integrated form, Brownian motion. Kernels have been used
for similar problems previously, albeit for the case of vector-valued data,
see e.g. \cite{CSG}. Proposals have been made to put similar ideas in to practice in the context of two-sided statistical tests determine whether two observed empirical measures on paths are drawn from the
same underlying distribution. For example \cite{CO-2018} work using the truncated signature kernel, while \cite{SLL} present an application based on the original signature kernel $\phi\equiv1$.

A formula for the expected Stratonovich signature of multivariate Brownian motion has been
known since the work of Fawcett \cite{fawcett} and Victoir \cite{LV-2004}. In the context of the problems
described above, we can take advantage of Fawcett's formula to prove what we
believe to be a novel identity, namely that for any continuous path $\gamma$
of bounded variation we have
\begin{equation}
\left\langle \mathbb{E}\left[  S\left(  \circ B\right)  _{0,s}\right]
,S\left(  \gamma\right)  _{0,t}\right\rangle _{\phi}=\cosh\left(  \rho
_{\sqrt{s/2}\gamma}\left(  t\right)  \right)  .\label{novel identity}%
\end{equation}
In this formula, $\rho_{\gamma}\left(  t\right)  $ is the hyperbolic distance
between the starting point and the end point of the hyperbolic development of
the path segment $\left.  \gamma\right\vert _{\left[  0,t\right]  }$, and
\[
\phi\left(  k\right)  :=\Gamma\left(  k/2+1\right)  :=\int_{0}^{\infty}%
x^{k/2}e^{-x}dx.
\]
When we realise hyperbolic space as a hyperboloid, the right hand side of
formula (\ref{novel identity}) can be obtained by solving a linear ordinary
differential equation. In the special case where $\gamma$ is piecewise linear,
this solution of the equation is a known product of matrices. These remarks
allow one to compute quantities like $d_{\phi}\left(  \mathcal{W},\nu\right)
$, where $\mathcal{W}$ denotes Wiener measure and $\nu$ is an empirical
measure on bounded variation paths. We note that the primary use of the
hyperbolic development in the study of signatures to date has been in
obtaining lower bounds for the study of signature asymptotics, see
\cite{HL2010} and \cite{BG2019}. In this context, the identity
(\ref{novel identity}) appears new, and it establishes a 
connection between the signature kernel and these broader topics. It seems plausible
that an additional benefit of (\ref{novel identity}) will be that it allows a
more analytic treatment of these other problems in a way that relies less on the
geometrical intricacies of hyperbolic space. 

If $\phi\equiv1$, we can use
Hankel's well-known representation for the reciprocal Gamma function as the
contour integral
\[
\frac{1}{\Gamma\left(  z\right)  }=\frac{1}{2\pi i}\oint_{H}w^{-z}e^{w}dw,
\]
where $H$ is Hankel's contour. Noting the similarity with (\ref{pairs}) we can
obtain the identity
\begin{equation}
\left\langle \mathbb{E}\left[  S\left(  \circ B\right)  _{0,s}\right]
,S\left(  \gamma\right)  _{0,t}\right\rangle _{\phi}=\frac{1}{2\pi i}\oint%
_{C}w^{-1}e^{w}\cosh\left(  \rho_{\sqrt{s/2w}\gamma}\left(  t\right)  \right)
dw,\label{hankel rep}%
\end{equation}
for an appropriate contour $C$. To make sense of this formula, we first need to make sense of the complex rescaling in the defining ODE for hyperbolic development. The numerical evaluation of contour integrals of the
form $\oint_{C}f\left(  w\right)  e^{w}dw$ is an active topic in numerical
integration, see \cite{ST-2007}, and we use these ideas to evaluate (\ref{hankel rep}). The same idea can be extended to cover
general $\phi.$

In the final two sections we consider examples which lend themselves to being
treated by the methods outlined above. A natural question is how to select an
appropriate $\phi$ for a given task and, the related question of how to
evaluate the performance of a given kernel against an alternative. To develop
this, we reverse the perspective taken above and use $d_{\phi}$ to define a
loss function%
\[
L_{\phi}\left(  \mathcal{W},\mu\right)  :=d_{\phi}\left(  \mathcal{W}%
,\mu\right)  _{\phi}^{2}%
\]
and, given a finite collection of paths $\left\{  \gamma_{1},...,\gamma
_{n}\right\},$ we consider the problem of minimising $L$ over the
set
\[
C_{n}=\left\{  \mu=%
{\textstyle\sum\nolimits_{i=1}^{n}}
\lambda_{i}\delta_{\gamma_{i}}:%
{\textstyle\sum\nolimits_{i=1}^{n}}
\lambda_{i}=1,\lambda_{i}\geq0\right\}  .
\]
Under some conditions on the support this optimisation problem will have a
unique solution $\mu^{\ast}$ which we can find. This gives us a way of
evaluating the similarity of a given finitely supported (possibly empirical)
measure $\mu$ to Wiener measure under the loss function induced by
$\left\langle \cdot,\cdot\right\rangle _{\phi}$ by comparing $L_{\phi}\left(
\mathcal{W},\mu\right)  $ and $L_{\phi}\left(  \mathcal{W},\mu^{\ast}\right)
$. For example, if the ratio $\frac{L_{\phi}\left(  \mathcal{W},\mu^{\ast
}\right)  }{L_{\phi}\left(  \mathcal{W},\mu\right)  }<\alpha<<1 $ then by an
appropriate selection of the threshold $\alpha$ one might decide that $\mu$
does not resemble Wiener measure. We do not give an extensive treatment of
examples, but to illustrate how these methods introduced above might be used
we consider two cases in detail:

\begin{enumerate}
\item Cubature measures of degree $N$ on Weiner space are finitely supported
measures which matched the expected iterated integrals of Brownian motion up
to and including degree $N.$ Explicit constructions are known in some cases,
see \cite{LV-2004}. By definition these measures will be optimal in the above
sense for any kernel induced by any $\phi$ with $\phi\left(  k\right)  =0$ for
$k\geq N.$ One might expect that they are close to optimal for smoother $\phi$
which still decay sufficiently fast.

\item We model radio frequency interference in sky-subtracted visibilities
radioastronomy as advocated by \cite{Wilensky2019} and consider two idealised
types of signal contamination:

\begin{itemize}
\item Narrow-band RFI measure across $n$ antennas. In this case the received
signals are $n$ linear superpositions of independent Brownian motions with a
single-frequency sinusoidal wave of a fixed amplitude.

\item Short duration high energy bursts. As a model for this we consider the
gerneralisation to the multivariate case of the example, originally considered
in the univariate setting in which the signal is given by $X_{t}%
=W_{t}+\epsilon\sqrt{\left(  t-U\right)  _{+}}$ for $t\in\left[  0,1\right]
,$where $\left(  W_{t}\right)  _{t\in\left[  0,1\right]  }$ is a Brownian
motion, $U$ is independent an uniformly distributed on $\left[  0,1\right]  $
and $\epsilon>0.$ The theoretical interest in this comes from the existence of
a critical parameter $\epsilon_{0}>0$ for which the law of $X$ is equivalent
to $\mathcal{W}$ if and only if $\epsilon<\epsilon_{0},$ see \cite{Davis-1984}, and which therefore gives an example that falls outside the scope of traditional maximum-likelihood-based approaches to the problem. 
\end{itemize}
\end{enumerate}

\section{Background on General Signature Kernels\label{sec-general-Signature-Kernels}}

Let $T\left(V\right)$ denote the algebra of tensor polynomials over
a finite dimensional vector space $V$ which consists of elements
of the form 
\[
a=\sum_{k=0}^{\infty}a_{k},\text{ }a_{k}\in V^{\otimes k}\text{ such that }a_{k}=0\text{ for all but finitely many }k,
\]
with the tensor product defined by
\[
ab=\sum_{k=0}^{\infty}\sum_{l=0}^{k}a_{l}b_{k-l}\text{ }
\]
where the product $V^{\otimes l}\times V^{\otimes\left(k-l\right)}\ni\left(c,d\right)\mapsto cd\in V^{\otimes k}$
is determined by $\left((v_{1}...v_{l}),(v_{l+1}...v_{k})\right)\mapsto v_{1}...v_{k}$
for $v_{1}...v_{k}\in V.$ We let $T\left(\left(V\right)\right)$
denote the space of formal tensor series, and $V^{\ast}$ denote the
(algebraic) dual space of $V.$ Then $T\left(V^{\ast}\right)$ is
the dual space of \ $T\left(\left(V\right)\right),$ and the signature
of a continuous bounded variation path $\gamma:[a,b]\rightarrow$
$V$ is the family of elements $\{S\left(\gamma\right)_{s,t}:s\leq t \in [a,b]\}$
in $T\left(\left(V\right)\right)$ determined inductively by
\begin{equation}
S\left(\gamma\right)_{s,t}\left(1\right)=1\text{ and }S\left(\gamma\right)_{s,t}\left(v_{1}...v_{k}\right)=\int_{s}^{t}S\left(\gamma\right)_{s,u}\left(v_{1}...v_{k-1}\right)d\gamma_{u}\left(v_{k}\right),\text{ with }v_{1},...,v_{k}\in V^{\ast}.\label{sig def}
\end{equation}
We will write 
\begin{equation}
S\left(\gamma\right)_{s,t}=1+\sum_{k=1}^{\infty}\int_{s<t_{1}<...<t_{k}<t}d\gamma_{t_{1}}...d\gamma_{t_{k}}\in T\left(\left(V\right)\right),\label{sig}
\end{equation}
and let 
\begin{equation}
\mathcal{S}=\{S\left(\gamma\right)_{s,t}:\gamma,s<t\} \subset T\left(\left(V\right)\right).\label{sig set}
\end{equation}
We consider dual pairs $\left(E,F\right),\,$where $E\ \ $and $F$
are two linear subspaces of $T\left(\left(V\right)\right).$ Recall
that this means that $\left(\cdot,\cdot\right):E\times F\rightarrow\mathbb{R}$
is a bilinear map such that the linear functionals $\left\{ \left(e,\cdot\right):e\in E\right\} \subset F^{\ast}$
and $\left\{ \left(\cdot,f\right):f\in F\right\} \subset E^{\ast}$
separate points in $F$ and $E$ respectively. We can identify $E$
and $F$ linear subspaces of the algebraic dual spaces $F^{\ast}$
and $E^{\ast}$ respectively.
\begin{defn}
\label{sig ker}Let $\left(E,F\right)$ be a dual pair as above. Suppose
that $\mathcal{S}\subset E\cap F$ where $S$ denotes the set of signatures
(\ref{sig set}). Then given two continuous paths $\gamma,\sigma:[a,b]\rightarrow V$
of bounded variation, we define the $\left(\cdot,\cdot\right)$-signature
kernel of $\gamma$ and $\sigma$ to be the function
\[
\left[a,b\right]\times\left[a,b\right]\ni\left(s,t\right)\mapsto\left(S\left(\gamma\right)_{a,s},S\left(\sigma\right)_{a,t}\right)=K_{\left(\cdot,\cdot\right)}^{\gamma,\sigma}\left(s,t\right).
\]
\end{defn}

\begin{rem}
This definition is not symmetric in general, i.e. it may hold that
$K_{\left(\cdot,\cdot\right)}^{\gamma,\sigma}\neq$ $K_{\left(\cdot,\cdot\right)}^{\sigma,\gamma}.$ 
\end{rem}

For this definition to be useful we need to demand more of the pairing
$\left(E,F\right).$ More exactly we need at least that their continuous duals
satisfy $F \subseteq E^{\prime} $ and $E \subseteq F^{\prime}.$ To go further still we
will need that they respect some of the algebraic structure on $T\left(\left(V\right)\right).$
The examples we will work are derived from a fixed but arbitrary inner
product $\left\langle \cdot,\cdot\right\rangle $ on $V.$ This gives
rise to the Hilbert-Schmidt inner product $\left\langle \cdot,\cdot\right\rangle _{k}$
on the $k$-fold tensor product spaces $V^{\otimes k}$ in a canonical
way. Then, by taking 
\[
\left\langle a,b\right\rangle _{\phi}:=\sum_{k=0}^{\infty}\phi\left(k\right)\left\langle a_{k},b_{k}\right\rangle _{k}
\]
for some weight function $\phi:
\mathbb{N}
\cup\left\{ 0\right\} \rightarrow
\mathbb{R}
_{+}$ we may define
$T_{\phi}\left(V\right)$ to be the Hilbert space obtained by completing $T\left(V\right)$
with respect to $\left\langle \cdot,\cdot\right\rangle _{\phi}.$ We equip $T_{\phi}\left(V\right)$ with the norm topology unless stated otherwise.
It is necessary to have a condition on $\phi$ which ensures that $\mathcal{S\subset}T_{\phi}\left(V\right).$
\begin{lem}
\label{lem-S-in-Tphi}Let $\phi:
\mathbb{N}
\cup\left\{ 0\right\} \rightarrow
\mathbb{R}
_{+}$ be such that for every $C>0$ the series $\sum_{k\in
\mathbb{N}
}C^{k}\phi\left(k\right)\left(k!\right)^{-2}$ is summable. Then $\mathcal{S\subset}T_{\phi}\left(V\right).$ 
\end{lem}

\begin{proof}
Let $\left\{ e_{i}:i=1,..,d\right\} $ be any orthonormal basis of
$V$ w.r.t. $\left\langle \cdot,\cdot\right\rangle ,$ and $\left\{ e_{I}^{\ast}:I=(i_{1},...,i_{k})\right\} $
the associated dual basis on $\left(V^{\ast}\right)^{\otimes k}.$
Then 
\begin{equation}
\left\vert \left\vert S\left(\gamma\right)_{s,t}\right\vert \right\vert _{\phi}^{2}=\sum_{k=0}^{\infty}\phi\left(k\right)\sum_{\left\vert I\right\vert =k}\left[S\left(\gamma\right)_{s,t}\left(e_{I}^{\ast}\right)\right]^{2},\label{series}
\end{equation}
and since 
\[
S\left(\gamma\right)_{s,t}\left(e_{I}^{\ast}\right)=\int_{s<u_{1}<u_{2}<...<u_{k}<t}d\left\langle e_{i_{1}},\gamma_{u_{1}}\right\rangle d\left\langle e_{i_{2}},\gamma_{u_{2}}\right\rangle ...d\left\langle e_{i_{k}},\gamma_{u}\right\rangle 
\]
we can estimate the summands in (\ref{series}) by 
\[
\sum_{\left\vert I\right\vert =k}\left[S\left(\gamma\right)_{s,t}\left(e_{I}^{*}\right)\right]^{2}\leq\frac{L_{s,t}\left(\gamma\right)^{2k}}{\left(k!\right)^{2}},\text{ where }L_{s,t}\left(\gamma\right):=\int_{s}^{t}\left\vert d\gamma_{u}\right\vert \text{ is the length of }\gamma.
\]
The summability condition then ensures that (\ref{series}) is finite. 
\end{proof}
This prompts the following condition.
\begin{condition}
\label{sum}The function $\phi:
\mathbb{N}
\cup\left\{ 0\right\} \rightarrow
\mathbb{R}
_{+}$ is such that the series $\sum_{k\in
\mathbb{N}
}C^{k}\phi\left(k\right)\left(k!\right)^{-2}$ is summable for every $C>0.$ 
\end{condition}

The next lemma describes examples of dual pairs $\left(E,F\right)$
of Hilbert spaces which fulfill the conditions in Definition \ref{sig ker}.
\begin{lem}
Let $\phi:
\mathbb{N}
\cup\left\{ 0\right\} \rightarrow
\mathbb{R}
_{+}$ and $\psi:
\mathbb{N}
\cup\left\{ 0\right\} \rightarrow
\mathbb{R}
_{+}$ be functions such that $\phi$ and $\psi^{-1}($i.e. $n\mapsto\psi\left(n\right)^{-1})$
satisfy the summability criterion of Condition \ref{sum}. In each
of the following cases $\left(E,F\right)$ is a dual pair which satisfies $F \subseteq E^{\prime} $ and $E \subseteq F^{\prime}.$
\begin{enumerate}
\item $E=T_{\phi}\left(V\right),F=T_{\phi}\left(V\right),\left(\cdot,\cdot\right)=\left\langle \cdot,\cdot\right\rangle _{\phi},$
\item $E=T_{\phi}\left(V\right),F=T_{\psi^{-1}}\left(V\right),\left(\cdot,\cdot\right)=\left\langle \cdot,\cdot\right\rangle _{\sqrt{\phi/\psi}}$. 
\end{enumerate}
\end{lem}

\begin{proof}
For notational ease we write $H_{\phi}$ for $T_{\phi}\left(V\right).$
In both cases Condition \ref{sum} ensures that $\mathcal{S}\subset E\cap F.$
In case 1, it is classical that $H_{\phi}^{\prime}=\{\left\langle h,\cdot\right\rangle _{\phi}:h\in H_{\phi}\},\,\ $while
for case 2 we have for \,$h\in H_{\phi}$ and $g\in H_{\psi^{-1}}$
we have that 
\[
\left\vert \left\langle h,g\right\rangle _{\sqrt{\phi/\psi}}\right\vert =\left\vert \sum_{k=0}^{\infty}\sqrt{\frac{\phi\left(k\right)}{\psi\left(k\right)}}\left\langle h_{k},g_{k}\right\rangle _{k}\right\vert \leq\left\vert \left\vert h\right\vert \right\vert _{\phi}\left\vert \left\vert g\right\vert \right\vert _{\psi^{-1}}
\]
hence $\{\left\langle h,\cdot\right\rangle _{\sqrt{\phi/\psi}}:h\in H_{\phi}\}\subseteq H_{\psi^{-1}}^{\prime}$.
By using the fact that $h$ is in $H_{\phi}$ if and only if $\tilde{h}:=\sqrt{\phi\psi}h:=\sum_{k}\sqrt{\phi\left(k\right)\psi\left(k\right)}h_{k}$
is in $H_{\psi^{-1}}$ we see that 
\[
\left\langle h,\cdot\right\rangle _{\sqrt{\phi/\psi}}=\left\langle \tilde{h},\cdot\right\rangle _{\psi^{-1}}
\]
so that $\{\left\langle h,\cdot\right\rangle _{\sqrt{\phi/\psi}}:h\in H_{\phi}\}=\{\left\langle h,\cdot\right\rangle _{\psi^{-1}}:h\in H_{\psi^{-1}}\}=H_{\psi^{-1}}^{\prime}.$ 
\end{proof}
Hereafter we will work almost entirely in the case $\left\langle T_{\phi}\left(V\right),T_{\phi}\left(V\right)\right\rangle _{\phi}$
in which the dual pair is the Hilbert space $T_{\phi}\left(V\right)$
with itself with pairing given by the inner product $\left\langle \cdot,\cdot\right\rangle _{\phi}.$
This leads to the following definition.
\begin{defn}
Let $\phi:
\mathbb{N}
\cup\left\{ 0\right\} \rightarrow
\mathbb{R}
_{+}$ satisfy Condition \ref{sum}. Given two continuous paths $\gamma,\sigma:[a,b]\rightarrow V$
of bounded variation, we define the $\phi$\textbf{-signature kernel}
of $\gamma$ and $\sigma$ to be the two-parameter function $K_{\phi}^{\gamma,\sigma}$
defined by 
\[
\left[a,b\right]\times\left[a,b\right]\ni\left(s,t\right)\mapsto\left\langle S\left(\gamma\right)_{a,s},S\left(\sigma\right)_{a,t}\right\rangle _{\phi}=:K_{\phi}^{\gamma,\sigma}\left(s,t\right).
\]
\end{defn}

\begin{rem}
It is straight forward to extend the discussion above to consider
general bilinear forms of signatures. If \ $\phi:
\mathbb{N}
\cup\left\{ 0\right\} \rightarrow
\mathbb{R}
$, then we can define a semi-definite inner product on $T\left(V\right)$
by
\[
\left\langle a,b\right\rangle _{\left\vert \phi\right\vert }:=\sum_{k=0}^{\infty}|\phi\left(k\right)|\left\langle a_{k},b_{k}\right\rangle _{k}.
\]
Let $N$ denote the linear subspace of $T\left(V\right)$ given by
the kernel of semi-norm $\left\vert \left\vert \cdot\right\vert \right\vert _{\left\vert \phi\right\vert }.$
Then we we can complete the quotient space $T\left(V\right)/N$ with
respect to inner product $\left\langle \cdot,\cdot\right\rangle _{\left\vert \phi\right\vert }$
and denote the resulting Hilbert space by $T_{|\phi|}\left(V\right).$
The bilinear form on $T\left(V\right)$ 
\begin{equation}
\left\langle a,b\right\rangle _{\phi}:=B_{\phi}\left(a,b\right):=\sum_{k=0}^{\infty}\phi\left(k\right)\left\langle a_{k},b_{k}\right\rangle _{k}\label{B on poly}
\end{equation}
extends to a continuous bilinear form on $T_{|\phi|}\left(V\right)$.
\ If $\phi$ is such that $\left\vert \phi\right\vert $ satisfies
Condition \ref{sum} then, as above, we define the $\phi$\textbf{-signature
kernel} of $\gamma$ and $\sigma$ to be the function $K_{\phi}^{\gamma,\sigma}:[a,b]\times\left[a,b\right]\rightarrow\mathbb{R}$
by 
\[
K_{\phi}^{\gamma,\sigma}\left(s,t\right):=\left\langle S\left(\gamma\right)_{a,s},S\left(\sigma\right)_{a,t}\right\rangle _{\phi}.
\]
This agrees with the previous definition whenever $\phi$ takes positive
values. 
\end{rem}

The following shifted weight functions arise naturally when doing calculus
on signature kernels.
\begin{defn}
Given a function $\phi:
\mathbb{N}
\cup\left\{ 0\right\} \rightarrow
\mathbb{R}
$ and $k\in
\mathbb{N}
$, we define the $k$-shift of $\phi$ to be the function $\phi_{+k}:
\mathbb{N}
\cup\left\{ 0\right\} \to\mathbb{R}$ determined by $\phi_{+k}\left(\cdot\right)=\phi\left(\cdot+k\right).$ 
\end{defn}

The next result is fundamental.
\begin{prop}
\label{pde result}Let $\gamma,\sigma:[a,b]\rightarrow V$ be two
continuous paths of bounded variation. Assume that the function \ $\phi:
\mathbb{N}
\cup\left\{ 0\right\} \rightarrow
\mathbb{R}
$ is such that $\left\vert \phi\right\vert $ and its $1$-shift $\left\vert \phi_{+1}\right\vert $
both satisfy Condition \ref{sum}. Then the $\phi$- and \ $\phi_{+1}$-
signature kernels of $\gamma$ and $\sigma$ are well defined and
are related by the two-parameter integral equation
\[
K_{\phi}^{\gamma,\sigma}\left(s,t\right)=\phi\left(0\right)+\int_{a}^{s}\int_{a}^{t}K_{\phi_{+1}}^{\gamma,\sigma}\left(u,v\right)\left\langle d\gamma_{u},d\sigma_{v}\right\rangle .
\]
\end{prop}

\begin{proof}
Well definedness of the two signature kernels follows from the summability
conditions. Unravelling the definitions and using (\ref{sig def})
gives 
\begin{align*}
K_{\phi}^{\gamma,\sigma}\left(s,t\right) & =\sum_{k=0}^{\infty}\phi\left(k\right)\sum_{\left\vert I\right\vert =k}S\left(\gamma\right)_{a,s}\left(e_{I}\right)S\left(\sigma\right)_{a,t}\left(e_{I}^{\ast}\right)\\
%
 & =\phi\left(0\right)+\sum_{k=1}^{\infty}\phi\left(k\right)\sum_{\left\vert I\right\vert =k-1}\int_{a}^{s}\int_{a}^{t}S\left(\gamma\right)_{a,u}\left(e_{I}^{\ast}\right)S\left(\sigma\right)_{a,v}\left(e_{I}^{\ast}\right)\left\langle d\gamma_{u},d\sigma_{v}\right\rangle \\
 & =\phi\left(0\right)+\int_{a}^{s}\int_{a}^{t}\sum_{k=0}^{\infty}\phi\left(k+1\right)\sum_{\left\vert I\right\vert =k}S\left(\gamma\right)_{a,u}\left(e_{I}^{\ast}\right)S\left(\sigma\right)_{a,v}\left(e_{I}^{\ast}\right)\left\langle d\gamma_{u},d\sigma_{v}\right\rangle \\
 & =\phi\left(0\right)+\int_{a}^{s}\int_{a}^{t}K_{\phi_{+1}}^{\gamma,\sigma}\left(s,t\right)\left\langle d\gamma_{u},d\sigma_{v}\right\rangle .
\end{align*}
\end{proof}
In the special case where $\phi$ is constant we see that the shift
$\phi_{+k}=\phi$ for every $k$ and therefore $K_{\phi}^{\gamma,\sigma}$
satisfies 
\[
K_{\phi}^{\gamma,\sigma}\left(s,t\right)=\phi\left(0\right)+\int_{a}^{s}\int_{a}^{t}K_{\phi}^{\gamma,\sigma}\left(s,t\right)\left\langle d\gamma_{u},d\sigma_{v}\right\rangle ,
\]
and in particular if $\gamma$ and $\sigma$ are differentiable and
$\phi\equiv1$ then we write $K_{\phi}^{\gamma,\sigma}=K^{\gamma,\sigma}$
and refer to it as the \textit{original signature kernel}. As was
first shown in \cite{clsw}, it\ solves the partial differential
equation 
\begin{equation}
\frac{\partial^{2}K^{\gamma,\sigma}\left(s,t\right)}{\partial s\partial t}=K^{\gamma,\sigma}\left(s,t\right)\left\langle \gamma_{s}^{\prime},\sigma_{t}^{\prime}\right\rangle \text{ on }\left[a,b\right]\times\left[a,b\right]\label{goursat}
\end{equation}
with boundary conditions \,$K\left(a,\cdot\right)\equiv K\left(\cdot,a\right)\equiv1.$
The same paper shows how the solution to (\ref{goursat}) can be approximated
numerically, and how the methodology extends to the case of rough
paths. The approximate solution can then be used to implement kernel
learning methods for classification or regression tasks based on time
series as mentioned in the introduction, see \cite{CO-2018,ko}.

It is self-evident from Proposition \ref{pde result} that for general
$\phi$ the function will not solve a PDE of the type (\ref{goursat}).
Nevertheless we can produce examples of different $\phi$ which do
by varying the inner-product $\left\langle \cdot,\cdot\right\rangle $
on the underlying vector space $V,$ or by scaling the inner product
on $T\left(V\right)$ homogeneously with respect the grading on $T(V)$. By the latter idea we mean that, for $\theta\in
\mathbb{R}
$ we can define $\delta_{\theta}:T\left(V\right)\rightarrow T\left(V\right)$
to be the unique algebra homomorphism which is determined by scalar
multiplication by $\theta$ on $V$ (i.e. $V\ni a\mapsto\theta a$),
then we have 
\begin{equation}
\delta_{\theta}a=\sum_{k=0}^{\infty}\theta^{k}a_{k},\,\ \,\text{if }a=\sum_{k=0}^{\infty}a_{k}\in T\left(V\right).\label{hom scaling}
\end{equation}
The following lemma explores the properties of $\delta_{\theta}$
when it is extended to a homogeneous linear map defined on (a subspace of)
the Hilbert space $T_{\phi}\left(V\right).$
\begin{lem}
Suppose $0\neq\theta\in
\mathbb{R}
$ and let $\phi:
\mathbb{N}
\cup\left\{ 0\right\} \rightarrow
\mathbb{R}
.$ Let $\theta\phi:
\mathbb{N}
\cup\left\{ 0\right\} \rightarrow
\mathbb{R}
$ denote the function defined by the pointwise product $\left(\theta\phi\right)\left(n\right)=\theta^{n}\phi\left(n\right)$ and let $\delta_{\theta}:T\left(V\right)\rightarrow T\left(V\right)$ be
the linear operator defined by (\ref{hom scaling}). Then:
\begin{enumerate}
\item For every $a,b\in T\left(V\right)$ we have the identity 
\begin{equation}
\left\langle \delta_{\theta}a,b\right\rangle _{\phi}=\left\langle a,\delta_{\theta}b\right\rangle _{\phi}=\left\langle a,b\right\rangle _{\theta\phi},\label{b id}
\end{equation}
which extends to $a,b\in T_{\left\vert \theta\phi\right\vert }\left(V\right)$.
The map $\delta_{\theta}$ extends uniquely to an isomorphism between
the Hilbert spaces $T_{\theta^{2}|\phi|}\left(V\right)$ and $T_{|\phi|}\left(V\right);$
\item For $\left\vert \theta\right\vert \leq1\,$ and $\phi>0$ we have
$T_{\phi}\left(V\right)\subseteq T_{\theta^{2}\phi}\left(V\right)$
and $\delta_{\theta}:T_{\phi}\left(V\right)$ $\rightarrow T_{\phi}\left(V\right)$
is a bounded self-adjoint linear 8operator with operator norm $\left\vert \left\vert \delta_{\theta}\right\vert \right\vert \leq1;$
\item For $\left\vert \theta\right\vert >1$ and $\phi>0,\,\ \delta_{\theta}$
is a linear operator $\delta_{\theta}:D(\delta_{\theta})$ $\rightarrow T_{\phi}\left(V\right)$
with domain $T_{\theta^{2}\phi}\left(V\right)\subseteq D(\delta_{\theta})\subset T_{\phi}\left(V\right).$
If furthermore $\phi$ satisfies Condition
\ref{sum}$,$ then $D(\delta_{\theta})$ is dense in $T_{\phi}\left(V\right)$
and $\delta_{\theta}$ is self-adjoint. 
\end{enumerate}
\end{lem}

\begin{proof}
For item 1, the identity (\ref{b id}) follows from (\ref{B on poly}).
The extension to the completion follows from the fact that $|\left\langle \delta_{\theta}a,b\right\rangle _{\phi}|\leq\left\vert \left\vert a\right\vert \right\vert _{\left\vert \theta\phi\right\vert }\left\vert \left\vert b\right\vert \right\vert _{\left\vert \theta\phi\right\vert }.$That
$\delta_{\theta}$ is an isometry between the pre-Hilbert spaces $(T\left(V\right)/N,\left\langle \cdot,\cdot\right\rangle _{\theta^{2}\left\vert \phi\right\vert })$
and $(T\left(V\right)/N,\left\langle \cdot,\cdot\right\rangle _{\left\vert \phi\right\vert })$
follows from (\ref{b id}) and the identity $\delta_{\theta}^{2}=\delta_{\theta^{2}}$:
\[
\left\langle \delta_{\theta}a,\delta_{\theta}b\right\rangle _{\left\vert \phi\right\vert }=\left\langle a,\delta_{\theta}^{2}b\right\rangle _{\left\vert \phi\right\vert }=\left\langle a,b\right\rangle _{\theta^{2}\left\vert \phi\right\vert },
\]
which extends to the completion $T_{\theta^{2}|\phi|}\left(V\right)$.
Surjectivity follows from the fact that $\delta_{\theta}\left(T\left(V\right)\right)=T\left(V\right)$
for any non-zero $\theta.$ For item 2, it is readily seen that $\left\vert \left\vert a\right\vert \right\vert _{\theta^{2}\phi}\leq\left\vert \left\vert a\right\vert \right\vert _{\phi}$
when $\left\vert \theta\right\vert \leq1$ for all $a\in T\left(V\right)$
and hence that $T_{\phi}\left(V\right)\subseteq T_{\theta^{2}\phi}\left(V\right).$
By item 1 we then have $\left\vert \left\vert \delta_{\theta}a\right\vert \right\vert _{\phi}\leq\left\vert \left\vert a\right\vert \right\vert _{\phi}$
which then extends to $T_{\phi}\left(V\right).$ Self-adjointness follows
from the identity
\begin{equation}
\left\langle \delta_{\theta}a,b\right\rangle _{\phi}=\sum_{k=0}^{\infty}\theta^{k}\phi\left(k\right)\left\langle a_{k},b_{k}\right\rangle _{k}=\left\langle a,\delta_{\theta}b\right\rangle _{\phi},\text{ for all }a,b\in T_{\phi}\left(V\right).\label{sa}
\end{equation}
Finally, for item 3 we observe that $T_{\theta^{2}\phi}\left(V\right)$
is a linear subspace of $T_{\phi}\left(V\right)$ and then that $\delta_{\theta}(T_{\theta^{2}\phi}\left(V\right))\subseteq$
$T_{\phi}\left(V\right)$ using item 1. If $ \phi $
satisfies Condition \ref{sum} then the domain of $\delta_{\theta}$
contains the linear span of the set of signatures $\mathcal{S}$ (recall
(\ref{sig set})) which is dense in $T_{\phi}\left(V\right)$. Self-adjointness
is again a consequence of (\ref{sa}). 
\end{proof}
As an immediate corollary we obtain the following result, which we
shall use repeatedly.
\begin{cor}
\label{cor-delta-kernel}Suppose $\theta\in
\mathbb{R}
$ and let $\phi:
\mathbb{N}
\cup\left\{ 0\right\} \rightarrow
\mathbb{R}
$ be such that $\left\vert \phi\right\vert $ satisfies Condition \ref{sum}
then 
\[
K_{\theta\phi}^{\gamma,\sigma}\left(s,t\right)=K_{\phi}^{\theta\gamma,\sigma}\left(s,t\right)=K_{\phi}^{\gamma,\theta\sigma}\left(s,t\right)
\]
for every $\left(s,t\right)\in\left[a,b\right]\times\left[a,b\right],$ where $\theta\gamma$ and $\theta \sigma$ denote the paths obtained by the pointwise multiplication of $\theta$ with $\gamma$ and $\sigma$ respectively. In particular if $\phi\equiv1$ then $K_{\theta\phi}^{\gamma,\sigma}:=$
$K_{\theta}^{\gamma,\sigma}$ satisfies 
\[
K_{\theta}^{\gamma,\sigma}\left(s,t\right)=1+\theta\int_{a}^{s}\int_{a}^{t}K_{\theta}^{\gamma,\sigma}\left(s,t\right)\left\langle d\gamma_{u},d\sigma_{v}\right\rangle .
\]
\end{cor}

\begin{proof}
We use the fact that $\delta_{\theta}S\left(\gamma\right)_{s,t}=S\left(\theta\gamma\right)_{s,t}$
and the previous lemma to observe that 
\[
K_{\theta\phi}^{\gamma,\sigma}\left(s,t\right)=\left\langle S\left(\gamma\right)_{a,s},S\left(\sigma\right)_{a,t}\right\rangle _{\theta\phi}=\left\langle \delta_{\theta}S\left(\gamma\right)_{a,s},S\left(\sigma\right)_{a,t}\right\rangle _{\phi}=K_{\phi}^{\theta\gamma,\sigma}\left(s,t\right).
\]
The fact that $K_{\phi}^{\theta\gamma,\sigma}\left(s,t\right)=K_{\phi}^{\gamma,\theta\sigma}\left(s,t\right)$
follows from the self-adjointness of $\delta_{\theta}.$ 
\end{proof}

\section{Representing General Signature Kernels\label{sec-Signature-Kernels}}

Let $\gamma$ be a continuous $V$-valued path of bounded variation.
Under the condition of Lemma \ref{lem-S-in-Tphi}, we can identify
the signature $S\left(\gamma\right)_{s,t}$ with an element of $T_{\phi}\left(V\right)$ and we can write 
\[
S\left(\gamma\right)_{s,t}:=\sum_{k=0}^{\infty}S\left(\gamma\right)_{s,t}^{k}\in T_{\phi}\left(V\right),
\]
where 
\[
S(\gamma)_{s,t}^{k}:=\int_{s<u_{1}<\cdots<u_{k}<t}d\gamma_{u_{1}}\cdots d\gamma_{u_{k}}\in V^{\otimes k},\ k\geq1
\]
and $S(\gamma)_{s,t}^{0}:=S(\gamma)_{s,t}(1)\equiv1$. 

Two properties in particular of the signature render it a good feature map.
First is its universality property; that is, provided one is careful about
definitions and topologies, continuous function on compact subspaces of
paths are uniformly approximable by linear functionals of the signature.
Central to this is a combination of the Stone-Weierstrass theorem and the
identity 
\[
S\left( \gamma \right) _{s,t}\left( f\right) S\left( \gamma \right)
_{s,t}\left( g\right) =S\left( \gamma \right) _{s,t}\left( f\shuffle g\right) 
\text{ for }f,g\in T\left( V^{\ast }\right) ,
\]%
where $f\shuffle g\in T\left( V^{\ast }\right) $ denotes the shuffle product
of the linear functionals $f$ and $g,$ see \cite{LCL}. The second property is
that signatures are characteristic in the sense that the expected signature
of a path-valued random variable will, under certain conditions,
characterise the law of that random variable, see \cite{HL2010,ko} for
more details. 

In the previous section we introduced the definition of the $\phi$-signature
kernel of continuous paths $\gamma$ and $\sigma$ to be the function
$K_{\phi}^{\gamma,\sigma}\left(s,t\right)$ . This amounts to reweighting
the terms in the signature to give more or less emphasis to high order
terms compared to the original signature kernel, i.e.
$\left\langle \cdot,\cdot\right\rangle _{\phi}$ for $\phi\equiv1$.
In the present section, we will build an approach to representing $\phi$-signature kernels in such a way 
that allows for efficient computation. The same idea is presented in multiple guises and then specialised within each case to 
yield particular examples. Before we present this method for $\phi$-signature kernels, we consider the error estimates which arise using a naive truncation-based approach. 

\subsection{Truncated Signature Kernels}

In this subsection, we give an error estimate of the truncated $\phi$-signature
kernel and the full $\phi$-signature kernel of two continuous bounded
variation paths. 

Let the truncated signature kernel be denoted
\begin{equation}
K_{\phi}^{(N)}(s,t):=\sum_{k=0}^{N}\phi(k)\left\langle S\left(\gamma\right)_{a,s}^{k},S\left(\sigma\right)_{a,t}^{k}\right\rangle _{k}=\sum_{k=0}^{N}\phi\left(k\right)\sum_{\left\vert I\right\vert =k}S\left(\gamma\right)_{a,s}\left(e_{I}^{*}\right)S\left(\sigma\right)_{a,t}\left(e_{I}^{*}\right).
\end{equation}
We have the following proposition.
\begin{prop}
\label{prop-error-bound}Let $\gamma,\sigma:[a,b]\rightarrow V$ be
two continuous paths of bounded variation. Assume that the function
$\phi:
\mathbb{N}
\cup\left\{ 0\right\} \rightarrow
\mathbb{R}
$ is such that $|\phi|$ satisfies Condition \ref{sum}, then the truncated
signature kernel $K_{\phi}^{(N)}(s,t)$ converges to the $\phi$-signature
kernel $K_{\phi}^{\gamma,\sigma}(s,t)$ when $N$ goes to infinity,
and the error bound is 
\begin{equation}
\left|K_{\phi}^{\gamma,\sigma}(s,t)-K_{\phi}^{(N)}(s,t)\right|\leq\sum_{k=N+1}^{\infty}|\phi(k)|(L_{s}(\gamma)L_{t}(\sigma))^{k}(k!)^{-2}
\end{equation}
where $L_{s}(\gamma)$ is the length of the path segment $\gamma|_{[a,s]}$.
\end{prop}

\begin{proof}
By the Cauchy-Schwarz inequality, we have
\begin{align*}
\left|K_{\phi}^{\gamma,\sigma}(s,t)-K_{\phi}^{(N)}(s,t)\right| & \leq\sum_{k=N+1}^{\infty}|\phi(k)|\left|\left\langle S\left(\gamma\right)_{a,s}^{k},S\left(\sigma\right)_{a,t}^{k}\right\rangle _{k}\right|\\
 & \leq\sum_{k=N+1}^{\infty}|\phi(k)|\left\vert \left\vert S\left(\gamma\right)_{a,s}^{k}\right\vert \right\vert _{k}\left\vert \left\vert S\left(\sigma\right)_{a,t}^{k}\right\vert \right\vert _{k}\\
 & =\sum_{k=N+1}^{\infty}|\phi(k)|\frac{(L_{s}(\gamma)L_{t}(\sigma))^{k}}{(k!)^{2}}.
\end{align*}
Since $|\phi|$ satisfies Condition \ref{sum}, the error goes to
0 as $N\to\infty$. 
\end{proof}
We analyse two concrete examples that we will revisit later using other methods.
\begin{itemize}
\item The first example takes $\phi$ to be 
\begin{equation}
\phi(k):=\left(\frac{k}{2}\right)!:=\Gamma\left(\frac{k}{2}+1\right)
\end{equation}
which plays an important role in Section \ref{sec-expected-signature-kernel} when we consider the expected signature of Brownian motion.
\item The second example is 
\begin{equation}
\phi(k)=\frac{\Gamma(m+1)\Gamma(k+1)}{\Gamma(k+m+1)}
\end{equation}
where $m\in\mathbb{R}_{+}$. The case when $m=0$, $\phi(k)\equiv1$
corresponds to the original signature kernel, while $m=1$ gives $\phi(k)=\frac{1}{k+1}$ which are the sequence of moments of a random variable which is uniformly distributed on $[0,1]$.
\end{itemize}
The following corollary specialises the previously-obtained error estimate to these cases.
\begin{cor}
Let $\gamma,\sigma:[a,b]\rightarrow V$ be two continuous paths of
bounded variation. Denote the length of the path segment $\gamma|_{[a,s]}$
as $L_{s}(\gamma)$. 

(1) The $\phi$-signature kernel under $\phi(k)=\left(\frac{k}{2}\right)!$
is well defined and there is a constant $C$ such that 
\begin{equation}
\left|K_{\phi}^{\gamma,\sigma}(s,t)-K_{\phi}^{(N)}(s,t)\right|\leq C\left(\frac{e}{2N+2}\right)^{N+1/2}e_{N+1}\left(L_{s}(\gamma)L_{t}(\sigma)\right)\label{eq-error-bound-1}
\end{equation}
where $e_{N+1}(x):=\sum_{k=N+1}^{\infty}\frac{x^{k}}{k!}$. 

(2) The $\phi$-signature kernel under $\phi(k)=\frac{\Gamma(m+1)\Gamma(k+1)}{\Gamma(k+m+1)}$
is well defined and the error bound is
\begin{equation}
\left|K_{\phi}^{\gamma,\sigma}(s,t)-K_{\phi}^{(N)}(s,t)\right|\leq\frac{\Gamma(m+1)}{\left(L_{s}(\gamma)L_{t}(\sigma)\right)^{\frac{m}{2}}}I_{m}^{(N+1)}\left(2\sqrt{L_{s}(\gamma)L_{t}(\sigma)}\right)\label{eq-error-bound-2}
\end{equation}
in which $I_{m}^{(N+1)}\left(z\right):=\left(\frac{z}{2}\right)^{m}\sum_{k=N+1}^{\infty}\frac{\left(\frac{1}{4}z^{2}\right)^{k}}{\Gamma(k+m+1)\Gamma(k+1)}$
is the tail of the series defining the modified Bessel function $I_{m}\left(z\right)$
of the first kind of order $m$. 
\end{cor}

\begin{proof}
It is easy to see that these two functions $\phi$ satisfy Condition
\ref{sum}, which makes sure that the $\phi$-signature kernels are
well defined. For the error bound (\ref{eq-error-bound-1}), by the
Stirling's approximation, there exist two constants $C_{1},C_{2}$
such that 
\[
C_{1}x^{x+\frac{1}{2}}e^{-x}\leq x!\leq C_{2}x^{x+\frac{1}{2}}e^{-x},\ \forall\ x>0.
\]
Then we have 
\[
\frac{\left(\frac{k}{2}\right)!}{k!}\leq\frac{C_{2}}{\sqrt{2}C_{1}}\left(\frac{e}{2k}\right)^{\frac{k}{2}}
\]
and the sequence on the right hand side is decreasing. Let $C=\frac{C_{2}}{\sqrt{2}C_{1}}$
and combine Proposition \ref{prop-error-bound}, it is easy to show
the error bound (\ref{eq-error-bound-1}). 

For the error bound (\ref{eq-error-bound-2}), since the modified
Bessel function $I_{m}\left(2\sqrt{L_{s}(\gamma)L_{t}(\sigma)}\right)$
of the first kind of order $m$ is defined by the series
\[
I_{m}\left(2\sqrt{L_{s}(\gamma)L_{t}(\sigma)}\right)=\left(L_{s}(\gamma)L_{t}(\sigma)\right)^{\frac{m}{2}}\sum_{k=0}^{\infty}\frac{(L_{s}(\gamma)L_{t}(\sigma))^{k}}{\Gamma(k+m+1)\Gamma(k+1)},
\]
the error bound follows from Proposition \ref{prop-error-bound}.
\end{proof}

\subsection{General Signature Kernels by Randomisation}

We now show how $\phi$-signature kernels can be represented, under suitable integrability conditions, as the average of rescaled PDE solutions whenever the sequence $\{\phi(k):k=0,1,...\}$ coincides with the sequence of moments of a random
variable. This representation consolidates the connection between the original and the $\phi$-signature kernels in these cases. The connection is captured in the following result. 
\begin{prop}
\label{thm-random-kernel}Suppose $\pi$ is a random variable with
finite moments of all orders and let the functions 
\begin{equation}
\phi(k)=\mathbb{E}[\pi^{k}]\ \text{and}\ \psi(k)=\mathbb{E}[|\pi|^{k}],\ \forall k\geq0.
\end{equation}
We assume that $\psi$ satisfies Condition \ref{sum}. Then the $\phi$-signature
kernel $K_{\phi}^{\gamma,\sigma}(s,t)$ of continuous bounded variation
paths $\gamma$ and $\sigma$ is well defined and 
\begin{equation}
K_{\phi}^{\gamma,\sigma}(s,t)=\mathbb{E}_{\pi}\left[K^{\pi\gamma,\sigma}(s,t)\right]=\mathbb{E}_{\pi}\left[K^{\gamma,\pi\sigma}(s,t)\right].\label{eq-random-kernel}
\end{equation}
\end{prop}

\begin{proof}
Since $|\phi|$ satisfies Condition \ref{sum}, which follows from
the condition of $\psi$, the $\phi$-signature kernel $K_{\phi}^{\gamma,\sigma}(s,t)$
is well defined. Furthermore, $\psi$ satisfies Condition \ref{sum},
by Fubini theorem, we have

\[
\begin{aligned}K_{\phi}^{\gamma,\sigma}(s,t) & =\sum_{k=0}^{\infty}\mathbb{E}\left[\pi^{k}\right]\left\langle S\left(\gamma\right)_{a,s}^{k},S\left(\sigma\right)_{a,t}^{k}\right\rangle _{k}\\
 & =\mathbb{E}\left[\sum_{k=0}^{\infty}\pi^{k}\left\langle S\left(\gamma\right)_{a,s}^{k},S\left(\sigma\right)_{a,t}^{k}\right\rangle _{k}\right]\\
 & =\mathbb{E}\left[\sum_{k=0}^{\infty}\left\langle S\left(\pi\gamma\right)_{a,s}^{k},S\left(\sigma\right)_{a,t}^{k}\right\rangle _{k}\right]\\
 & =\mathbb{E}\left[K^{\pi\gamma,\sigma}(s,t)\right].
\end{aligned}
\]
We conclude the proof. 
\end{proof}
\begin{rem}
If the random variable $\pi$ has a known probability density function,
the expectation in equation (\ref{eq-random-kernel}) can be calculated
by numerical methods such as Monte Carlo method or Gaussian quadrature
procedure. 
\end{rem}

The corollary below gives two specialisations of this result to the cases described earlier. 
\begin{cor}
\label{cor-sigkernel-exp}Let $\gamma,\sigma:[a,b]\rightarrow V$
be two continuous paths of bounded variation. 

(1) The $\phi$-signature kernel under $\phi(k)=\left(\frac{k}{2}\right)!$
satisfies  
\begin{equation}
K_{\phi}^{\gamma,\sigma}(s,t)=\mathbb{E}_{\pi}\left[K^{\pi^{1/2}\gamma,\sigma}(s,t)\right]=\mathbb{E}_{\pi}\left[K^{\gamma,\pi^{1/2}\sigma}(s,t)\right],\label{eq-f-kernel}
\end{equation}
where $\pi\sim\text{Exp}(1)$ is an exponentially distributed random
variable with intensity $1$. 

(2) The $\phi$-signature kernel $K_{\phi}^{\gamma,\sigma}(s,t)$
under $\phi(k)=\frac{\Gamma(m+1)\Gamma(k+1)}{\Gamma(k+m+1)}$ satisfies
equation (\ref{eq-random-kernel}) where $\pi\sim B(1,m)$ is a Beta-distributed
random variable. 
\end{cor}

\begin{proof}
For (1), we need to show that $\phi$ is all the moments of the random
variable $\pi^{1/2}$. Since $\pi\sim\text{Exp}(1)$, we have 
\[
\mathbb{E}\left[\text{\ensuremath{\pi}}^{k/2}\right]=\int_{0}^{\infty}x^{k/2}e^{-x}dx=\Gamma\left(\frac{k}{2}+1\right)=\phi(k).
\]
The equation (\ref{eq-f-kernel}) then follows from Theorem \ref{thm-random-kernel}.
For (2), since the random variable $\pi$ is Beta distributed, i.e.
$\pi\sim\textrm{Beta}(1,m)$, then the moments of $\pi$ are
\[
\mathbb{E}[\pi^{k}]=\frac{B(k+1,m)}{B(1,m)}=\frac{\Gamma(k+1)\Gamma(m+1)}{\text{\ensuremath{\Gamma}}(k+m+1)}=\phi(k).
\]
We conclude the proof. 
\end{proof}
The motivation for the representation (\ref{eq-random-kernel}) is
that we can design efficient and accurate computational methods to
compute the $\phi$-signature kernels. We will give details on the
Gaussian quadrature methods for the $\phi$-signature kernel in Section
\ref{subsec-Quadrature-Error-Estimates} below. 

\subsection{General Signature Kernels by Fourier Series}

We now extend the earlier discussion so that $\phi:
\mathbb{Z}
\rightarrow$ $
\mathbb{C}
$ is a complex-valued function. We consider the blinear form defined by the two-sided summation
\[
\left\langle a,b\right\rangle _{\phi}:=B_{\phi}\left(a,b\right):=\sum_{k=-\infty}^{\infty}\phi\left(k\right)\left\langle a_{|k|},b_{|k|}\right\rangle _{|k|},
\]
and the corresponding function 
\[
K_{\phi}^{\gamma,\sigma}\left(s,t\right):=\left\langle S\left(\gamma\right)_{a,s},S\left(\sigma\right)_{a,t}\right\rangle _{\phi}.
\]
If the coefficients are the Fourier coefficients of some known periodic function $f$ then the idea of the previous proposition can be applied to again derive a representation of $K_{\phi}^{\gamma,\sigma}$. The following result describes the needed conditions.
\begin{prop}
Suppose that $\gamma$ and $\sigma$ are continuous paths of bounded
$1$-variation. Let $\phi:
\mathbb{Z}
\rightarrow$ $
\mathbb{C}
$ be as above, and write $\phi_{k}:=\phi\left(k\right).$ Assume that
$\left\{ \phi_{k}:k\in
\mathbb{N}
\right\} $ are the Fourier coefficients of some bounded integrable function
$f:(-\pi,\pi)\rightarrow
\mathbb{C}
,$ i.e.
\[
f=\sum_{k=-\infty}^{\infty}\phi_{k}e_{k},\text{ with }e_{k}\left(x\right):=e^{ikx}.
\]
Then for all $\left(s,t\right)\in\left[a,b\right]\times\left[a,b\right]$
we have 
\begin{equation}
K_{\phi}^{\gamma,\sigma}\left(s,t\right)=\frac{1}{2\pi}\int_{-\pi}^{\pi}\bar{K}_{x}^{\gamma,\sigma}\left(s,t\right)f\left(x\right)dx-\phi_{0},
\end{equation}
where
\[
\bar{K}_{x}^{\gamma,\sigma}\left(s,t\right):=K^{\exp\left(-ix\right)\gamma,\sigma}\left(s,t\right)+K^{\exp\left(ix\right)\gamma,\sigma}\left(s,t\right).
\]
\end{prop}

\begin{proof}
Fixing $\left(s,t\right)$, we have for every $x\in\left(-\pi,\pi\right)$
that 
\[
K^{\exp\left(\pm ix\right)\gamma,\sigma}\left(s,t\right)=\sum_{k=0}^{\infty}e_{\pm k}\left(x\right)\left\langle S\left(\gamma\right)_{a,s}^{k},S\left(\sigma\right)_{a,t}^{k}\right\rangle _{k}=:\sum_{k=0}^{\infty}e_{\pm k}\left(x\right)c_{k}.
\]
The basic estimate $\left\vert c_{k}\right\vert \leq L_{\gamma}^{k}L_{\sigma}^{k}/\left(k!\right)^{2}$
where $L_{\gamma}$ is the length of the path $\gamma$ ensures that
$\sum_{k=0}^{N}c_{k}e_{\pm k}\left(\cdot\right)f\left(\cdot\right)$
converges uniformly to the series $\sum_{k=0}^{\infty}c_{k}e_{\pm k}\left(\cdot\right)f\left(\cdot\right)$
and hence 
\[
\frac{1}{2\pi}\int_{-\pi}^{\pi}K^{\exp\left(\pm ix\right)\gamma,\sigma}\left(s,t\right)f\left(x\right)dx=\frac{1}{2\pi}\sum_{k=0}^{\infty}c_{k}\int_{-\pi}^{\pi}e_{\pm k}\left(x\right)f\left(x\right)dx=\sum_{k=0}^{\infty}c_{k}\phi_{\mp k}.
\]
It follows that 
\begin{align*}
 & \frac{1}{2\pi}\int_{-\pi}^{\pi}\left[K^{\exp\left(-ix\right)\gamma,\sigma}\left(s,t\right)+K^{\exp\left(ix\right)\gamma,\sigma}\left(s,t\right)\right]f\left(x\right)dx\\
 & =\sum_{k=-\infty}^{\infty}c_{|k|}\phi_{k}+c_{0}\phi_{0}=K_{\phi}^{\gamma,\sigma}\left(s,t\right)+\phi_{0},
\end{align*}
as required. 
\end{proof}
\begin{rem}
Note that $\mathcal{R}K_{x}^{\gamma,\sigma}\left(s,t\right):=\operatorname{Re}K^{\exp\left(ix\right)\gamma,\sigma}\left(s,t\right)$
so that 
\[
\mathcal{R}K_{x}^{\gamma,\sigma}\left(s,t\right)=\sum_{k=0}^{\infty}\cos kx\left\langle S\left(\gamma\right)_{a,s}^{k},S\left(\sigma\right)_{a,t}^{k}\right\rangle _{k}.
\]
Together with $\mathcal{I}K_{x}^{\gamma,\sigma}\left(s,t\right):=\operatorname{Im}K^{\exp\left(ix\right)\gamma,\sigma}\left(s,t\right)$
it solves the 2-dimensional PDE
\[
\frac{\partial^{2}}{\partial s\partial t}\left(\begin{array}{c}
\mathcal{R}K_{x}^{\gamma,\sigma}\left(s,t\right)\\
\mathcal{I}K_{x}^{\gamma,\sigma}\left(s,t\right)
\end{array}\right)=\left(\begin{array}{cc}
\cos x & -\sin x\\
\sin x & \cos x
\end{array}\right)\left(\begin{array}{c}
\mathcal{R}K_{x}^{\gamma,\sigma}\left(s,t\right)\\
\mathcal{I}K_{x}^{\gamma,\sigma}\left(s,t\right)
\end{array}\right)\left\langle \gamma_{s}^{\prime},\sigma_{t}^{\prime}\right\rangle .
\]
\end{rem}

\begin{cor}
Special cases of the above result include:
\begin{enumerate}
\item If $\phi_{k}=0$ for $k<0$ then 
\[
K_{\phi}^{\gamma,\sigma}\left(s,t\right)=\frac{1}{2\pi}\int_{-\pi}^{\pi}K^{\exp\left(-ix\right)\gamma,\sigma}\left(s,t\right)f\left(x\right)dx.
\]
\item (Real Fourier series) Suppose
\[
f=a_{0}+\sum_{k=1}^{\infty}a_{k}c_{k}+\sum_{k=1}^{\infty}b_{k}s_{k},\text{ where }c_{k}\left(\cdot\right):=\cos\left(k\cdot\right),~s_{k}\left(\cdot\right):=\sin\left(k\cdot\right)
\]
with $\left\{ a_{k}\right\} $ and $\left\{ b_{k}\right\} $ real
sequences. If
\begin{equation}
\left\langle p,q\right\rangle _{\phi}:=\sum_{k=0}^{\infty}a_{k}\left\langle p_{k},q_{k}\right\rangle _{k},
\end{equation}
then 
\[
K_{\phi}^{\gamma,\sigma}\left(s,t\right)=\frac{1}{\pi}\int_{-\pi}^{\pi}\mathcal{R}K_{x}^{\gamma,\sigma}\left(s,t\right)f\left(x\right)dx-a_{0}.
\]
\end{enumerate}
\end{cor}

In using this result the function $f$ should be chosen that the integral
can easily approximated numerically.
\begin{example}
The following simple examples illustrate the scope of these ideas.
\begin{enumerate}
\item The function $f\left(x\right)=x^{2}$ has the Fourier series $f=\sum_{k=-\infty}^{\infty}\phi_{k}e_{k}$
on $\left[-\pi,\pi\right]$ where
\[
\phi_{k}=\frac{4\left(-1\right)^{k}}{k^{2}},\phi_{0}=\frac{\pi^{2}}{3},
\]
and we obtain the identity
\[
\sum_{k=1}^{\infty}\frac{4\left(-1\right)^{k}}{k^{2}}\left\langle S\left(\gamma\right)_{a,s}^{k},S\left(\sigma\right)_{a,t}^{k}\right\rangle _{k}=\frac{1}{2\pi}\int_{-\pi}^{\pi}\mathcal{R}K_{x}^{\gamma,\sigma}\left(s,t\right)x^{2}dx-\frac{\pi^{2}}{3}.
\]
\item The periodic function $f\left(x\right)=e^{\cos x}\cos\left(\sin x\right)$
has Fourier series 
\[
f\left(x\right)=\sum_{k=0}^{\infty}\frac{1}{k!}\cos\left(kx\right)
\]
and so 
\[
\sum_{k=0}^{\infty}\frac{1}{k!}\left\langle S\left(\gamma\right)_{a,s}^{k},S\left(\sigma\right)_{a,t}^{k}\right\rangle _{k}=\frac{1}{\pi}\int_{-\pi}^{\pi}\mathcal{R}K_{x}^{\gamma,\sigma}\left(s,t\right)e^{\cos x}\cos\left(\sin x\right)dx-1.
\]
\item The Jacobi theta function is the $1$-periodic function 
\[
\theta\left(z;\tau\right)=1+2\sum_{k=1}^{\infty}e^{i\pi\tau k^{2}}\cos\left(2\pi kz\right),
\]
hence if we define $f\left(x;u\right):=\theta\left(\frac{x}{2\pi};\frac{iu}{\pi}\right),$
then $f\left(\cdot;u\right)=1+\sum_{k=1}^{\infty}e^{-uk^{2}}c_{k}$
and 
\[
\sum_{k=0}^{\infty}e^{-uk^{2}}\left\langle S\left(\gamma\right)_{a,s}^{k},S\left(\sigma\right)_{a,t}^{k}\right\rangle _{k}=\frac{1}{\pi}\int_{-\pi}^{\pi}\mathcal{R}K_{x}^{\gamma,\sigma}\left(s,t\right)f\left(x;u\right)dx-1.
\]
\end{enumerate}
\end{example}

\subsection{General Signature Kernels by Integral Transforms}

The main idea of the previous subsection was to look for a function
$f$ with Fourier series $\sum_{k\in
\mathbb{Z}
}\phi\left(k\right)e_{k}.$ If such a function can be found, then we can calculate the bilinear
form $B_{\phi}$ evaluated at a pair of signatures. The difficulty
with this approach is that such a function may not exist in some cases of interest,
e.g. $\phi\left(k\right)=k^{-1/2},$ $\phi\left(k\right)=k!$ etc.
To simplify we forego the two-sided summation, and re-define 
\[
\left\langle a,b\right\rangle _{\phi}:=B_{\phi}\left(a,b\right):=\sum_{k=0}^{\infty}\phi\left(k\right)\left\langle a_{k},b_{k}\right\rangle _{k},
\]
where $\phi:
\mathbb{R}
\rightarrow
\mathbb{C}
$ is now defined on $\mathbb{R}$. We assume that $\phi$ is the integral
of a function $r:\mathbb{R\times R\rightarrow}
\mathbb{C}
$ against a finite signed Borel measure $\mu$ on $\mathbb{R}$ such
that 
\begin{equation}
\phi\left(u\right)=\int_{C}r\left(u,z\right)\mu\left(dz\right);\text{ where }r\left(u,z\right)=g\left(z\right)^{\alpha u}\text{ }\in
\mathbb{C}
\text{ for }\alpha\in
\mathbb{R}.
\label{it}
\end{equation}

\begin{example}
\label{examples}We will consider three principal examples:
\begin{enumerate}
\item Fourier-Stieltjes transform: $C=
\mathbb{R}
,\ g\left(z\right)=e^{-2\pi iz},\ \alpha=1$, i.e. $\phi\left(u\right)=\hat{\mu}\left(u\right):=\int_{
\mathbb{R}
}e^{-2\pi iuz}\mu\left(dz\right);$
\item Laplace-Stieltjes transform: $C=\left(0,\infty\right),\ g\left(z\right)=e^{-z},\ \alpha=1$,
i.e. $\phi\left(u\right)=\tilde{\mu}\left(u\right):=$ $\int_{0}^{\infty}e^{-uz}\mu\left(dz\right);$
\item Mellin-Stieltjes transform: $C=\left(0,\infty\right),\ g\left(z\right)=z,\ \alpha=1$,
i.e. $\phi\left(u\right)=\mu_{\text{Mel}}\left(u+1\right)=\int_{0}^{\infty}z^{u}\mu\left(dz\right)$,
$\operatorname{Re}u>-1.$ 
\end{enumerate}
\end{example}

In the general case we can expect - under reasonable assumptions - that the
integral representation can be used to justify the calculation
\begin{align}
\left\langle S\left(\gamma\right)_{a,s},S\left(\sigma\right)_{a,t}\right\rangle _{\phi} & =\sum_{k=0}^{\infty}\int_{C}g\left(z\right)^{\alpha k}\mu\left(dz\right)\left\langle S\left(\gamma\right)_{a,s}^{k},S\left(\sigma\right)_{a,t}^{k}\right\rangle _{k}\nonumber \\
 & =\int_{C}\sum_{k=0}^{\infty}\left\langle S\left(g\left(z\right)^{\alpha}\gamma\right)_{a,s}^{k},S\left(\sigma\right)_{a,t}^{k}\right\rangle _{k}\mu\left(dz\right)\\
 & =\int_{C}K^{g\left(z\right)^{\alpha}\gamma,\sigma}\left(s,t\right)\mu\left(dz\right).\nonumber 
\end{align}
again allowing us to reduce the calculation of the the bilinear form to
a weighted integral over PDE solutions. On this occasion integration is w.r.t. the measure $\mu$ and the rescaling is determined  
by the form of the kernel function $r$ in the integral transform relating $\mu$ and $\phi$.
\begin{thm}
\label{general prop}Let $\mu$ be a finite signed Borel measure $\mu$
on $\mathbb{R}$. Suppose that $\phi:
\mathbb{R}
\rightarrow
\mathbb{C}
$ is such that
\[
\phi\left(k\right)=\int_{C}r\left(k,z\right)\mu\left(dz\right)\in
\mathbb{C}
\,\text{, for all }k\in\mathbb{N}\cup\{0\}
\]
where $r\left(u,\cdot\right)$ is assumed to have the form $r\left(u,z\right)=$
$g\left(z\right)^{\alpha u}$ $\in
\mathbb{C}
$ for $\alpha\in
\mathbb{R}
$ and some function $g:\mathbb{C}\rightarrow
\mathbb{C}
.$ Let 
\[
\gamma:\left[a,b\right]\rightarrow V\text{ and }\sigma:\left[a,b\right]\rightarrow V
\]
be continuous paths of bounded $1$-variation with signatures $S\left(\gamma\right)$
and $S\left(\sigma\right)$ respectively. For every $\left(s,t\right)\in\left[a,b\right]\times\left[a,b\right]$
and $k\in\mathbb{N\cup}\left\{ 0\right\} $ define 
\[
a_{k}\left(s,t\right):=\left\langle S\left(\gamma\right)_{a,s}^{k},S\left(\sigma\right)_{a,t}^{k}\right\rangle _{k}.
\]
Assume for every $\left(s,t\right)\in\left[a,b\right]\times\left[a,b\right]$
that
\begin{enumerate}
\item the integral $\int_{C}\left\vert r\left(k,z\right)\right\vert \left\vert \mu\left(dz\right)\right\vert <\infty,$
and
\item the series $\sum_{k}a_{k}\left(s,t\right)\int_{C}\left\vert r\left(k,z\right)\right\vert \left\vert \mu\left(dz\right)\right\vert $
converges absolutely, 
\end{enumerate}
then 
\begin{equation}
\left\langle S\left(\gamma\right)_{a,s},S\left(\sigma\right)_{a,t}\right\rangle _{\phi}=\int_{C}K^{g\left(z\right)^{\alpha}\gamma,\sigma}\left(s,t\right)\mu\left(dz\right).\label{result}
\end{equation}
\end{thm}

\begin{rem}
Sufficient for item 2 is that $\sum_{k}\left(k!\right)^{-2}\int_{C}\left\vert h_{k}\left(z;s,t\right)\right\vert \left\vert dz\right\vert $
converges. 
\end{rem}

\begin{proof}
Assumptions 1 and 2 above ensure that Fubini's Theorem can be applied
to give
\[
\sum_{k=0}^{\infty}a_{k}\left(s,t\right)\int_{C}r\left(k,z\right)\mu\left(dz\right)=\int_{C}\sum_{k=0}^{\infty}a_{k}\left(s,t\right)r\left(k,z\right)\mu\left(dz\right),
\]
which can be seen to be the same as (\ref{result}) using the fact
$r\left(u,z\right)\equiv$ $g\left(z\right)^{\alpha u}.$ 
\end{proof}
\begin{cor}
For each of the three integral transforms in Example \ref{examples}
satisfying assumption 1 and 2 in the above theorem, we have (\ref{result}). 
\end{cor}

In a similar way we have the following results once again.
\begin{cor}
Let $\pi$ be a random variable with finite moments of all orders
and
\[
\phi\left(k\right)=\mathbb{E}\left[\pi^{k}\right]\ \text{and}\ \psi\left(k\right)=\mathbb{E}\left[|\pi|^{k}\right]\text{ for all }k\in\mathbb{N}\cup\left\{ 0\right\} 
\]
such that $\psi$ satisfies Condition \ref{sum}. Then
\[
K_{\phi}^{\gamma,\sigma}(s,t)=\mathbb{E}_{\pi}\left[K^{\pi\gamma,\sigma}(s,t)\right]=\mathbb{E}_{\pi}\left[K^{\gamma,\pi\sigma}(s,t)\right].
\]
\end{cor}

\begin{proof}
Let $F$ be the distribution function of $\pi$. Apply Theorem \ref{general prop}
with $\mu=dF$ and $r\left(u,z\right)=z^{u}.$ 
\end{proof}
\begin{example}
These examples illustrate these results
\begin{enumerate}
\item For any $\beta>-1,$ the function $\phi\left(u\right)=\Gamma\left(u+\beta+1\right)=\int_{0}^{\infty}x^{u}x^{\beta}e^{-x}dx$
is the Mellin transform of $x^{\beta}e^{-x}$. Therefore, we have
\[
\sum_{k=0}^{\infty}\Gamma\left(k+\beta+1\right)\left\langle S\left(\gamma\right)_{a,s}^{k},S\left(\sigma\right)_{a,t}^{k}\right\rangle _{k}=\int_{0}^{\infty}K^{x\gamma,\sigma}\left(s,t\right)x^{\beta}e^{-x}dx.
\]
\item Suppose $\pi$ is a random variable, the expectation can be computed
in the following cases:
\begin{enumerate}
\item if $\pi$ is uniformly distributed on $\left[0,1\right],$ then it
equals 
\[
\sum_{k=0}^{\infty}\frac{1}{k+1}\left\langle S\left(\gamma\right)_{a,s}^{k},S\left(\sigma\right)_{a,t}^{k}\right\rangle _{k}=\int_{0}^{1}K^{x\gamma,\sigma}\left(s,t\right)dx;
\]
\item if $\pi$ has the Arcsine$\left(-1,1\right)$-distribution, i.e. $F_{\pi}\left(x\right)=\frac{2}{\pi}\arcsin\left(\sqrt{\frac{1+x}{2}}\right)$
, then: 
\begin{equation}
\sum_{k=0}^{\infty}
\prod_{r=0}^{2k-1}
\frac{2r+1}{2r+2}\left\langle S\left(\gamma\right)_{a,s}^{k},S\left(\sigma\right)_{a,t}^{k}\right\rangle _{k}=\frac{1}{\pi}\int_{-1}^{1}\frac{K^{x\gamma,\sigma}\left(s,t\right)}{\sqrt{1-x^{2}}}dx;\label{acrsin}
\end{equation}
\item if $\pi$ has the Beta$\left(\alpha,\beta\right)$-distribution, then:
\[
\sum_{k=0}^{\infty}
\prod_{r=0}^{k-1}
\frac{\alpha+\beta}{\alpha+\beta+r}\left\langle S\left(\gamma\right)_{a,s}^{k},S\left(\sigma\right)_{a,t}^{k}\right\rangle _{k}=\frac{1}{B(\alpha,\beta)}\int_{0}^{1}K^{x\gamma,\sigma}(s,t)x^{\alpha-1}\left(1-x\right)^{\beta-1}dx.
\]
\end{enumerate}
\end{enumerate}
\end{example}

\section{Computing General Signature Kernels\label{subsec-Quadrature-Error-Estimates}}
The usefulness of the formulae in the last section depend on being able
to numerically approximate integrals such as 
\[
\int_{a}^{b}f\left(x\right)w\left(x\right)dx
\]
where $[a,b]\subseteq
\mathbb{R}
$, $w\in L^{1}\left(\left(a,b\right)\right)$ is a weight function,
which for the moment we assume to be positive. In the examples considered
the function $f$ to be integrated will be a scaling of the signature
kernel PDE, typically we will have 
\[
f\left(x\right)=K^{x\gamma,\sigma}(s,t).
\]
The classical approach to such approximations is to use a Gaussian
Quadrature Rule, see e.g. \cite{suli}

For a general weight function, suppose that $\mathcal{P=}\left\{ p_{n}:n\in
\mathbb{N}
\cup\left\{ 0\right\} \right\} $ is a system of orthogonal polynomials w.r.t. the weight function
$w\,$ over $\left(a,b\right)$; that is $\deg\left(p_{n}\right)=n$
and $\left\langle p_{n},p_{m}\right\rangle _{w}=\int_{a}^{b}p_{m}p_{n}wdx=0$
for $n\neq m.$ Then the quadrature points $x_{k}$, $k=0,1,...,n$
are the zeros of the polynomial $p_{n+1}$, the corresponding quadrature
weights are 
\[
w_{k}:=\int_{a}^{b}w\left(x\right)
\prod_{i=0,i\neq k}^{n}
\left(\frac{x-x_{i}}{x_{k}-x_{i}}\right)^{2}dx
\]
and the quadrature rule is the approximation
\[
\int_{a}^{b}f\left(x\right)w\left(x\right)dx\approx\sum_{k=0}^{n}w_{k}f\left(x_{k}\right).
\]
The approximation is exact if $f$ is a polynomial with $\deg\left(f\right)\leq2n+1.$
If $f$ is assumed to be $C^{2n+2},$ then the error in the quadrature
rule can be approximated by the basic estimate \cite{suli} 
\begin{equation}
\left\vert \int_{a}^{b}f\left(x\right)w\left(x\right)dx-\sum_{k=0}^{n}w_{k}f\left(x_{k}\right)\right\vert \leq\frac{f^{\left(2n+2\right)}\left(\xi\right)}{\left(2n+2\right)!}\int_{a}^{b}w\left(x\right)\pi_{n+1}\left(x\right)^{2}dx,\label{estimate}
\end{equation}
where $\xi\in\left(a,b\right)$ and 
\[
\pi_{n+1}\left(x\right)=
\prod_{i=0}^{n}
\left(x-x_{i}\right)
\]
is the monic poynomial obtained by dividing $p_{n+1}$ by its leading
coefficient. In view of the bound (\ref{estimate}) we have the following 
\begin{lem}
Define $f\left(x\right):=K^{x\gamma,\sigma}(s,t)$ for $x\in
\mathbb{R}
.$ Then $f$ is infinitely differentiable and, for every $k\in
\mathbb{N}
$ , its $k$th derivative is given by 
\begin{equation}
f^{\left(k\right)}\left(x\right)=\sum_{l=0}^{\infty}x^{l}\frac{\left(l+k\right)!}{l!}\left\langle S\left(\gamma\right)_{a,s}^{l+k},S\left(\sigma\right)_{a,t}^{l+k}\right\rangle _{l+k}.\label{kth deriv}
\end{equation}
In particular, we have the estimate 
\begin{equation}
\left\vert f^{\left(k\right)}\left(x\right)\right\vert \leq\frac{L_{s}\left(\gamma\right)^{k/2}L_{t}\left(\sigma\right)^{k/2}}{|x|^{k/2}}I_{k}\left(2\sqrt{|x|L_{s}\left(\gamma\right)L_{t}\left(\sigma\right)}\right),\label{bound}
\end{equation}
where $L_{s}(\gamma)$ is the length of the path segment $\gamma|_{[a,s]}$
and $I_{k}$ is the modified Bessel function of the first kind of
order $k.$ 
\end{lem}

\begin{proof}
Differentiablity is a simple argument on term-by-term differentiation
of power series. Applying this argument $k$ times results in the
formula (\ref{kth deriv}). The bound (\ref{bound}) can be obtained
by the elementary estimate
\begin{align*}
\left\vert f^{\left(k\right)}\left(x\right)\right\vert  & \leq\sum_{l=0}^{\infty}|x|^{l}\frac{\left(l+k\right)!}{l!}\frac{L_{s}\left(\gamma\right)^{l+k}L_{t}\left(\sigma\right)^{l+k}}{\left(l+k\right)!^{2}}\\
 & =\frac{L_{s}\left(\gamma\right)^{k/2}L_{t}\left(\sigma\right)^{k/2}}{|x|^{k/2}}I_{k}\left(2\sqrt{|x|L_{s}\left(\gamma\right)L_{t}\left(\sigma\right)}\right).
\end{align*}
\end{proof}
For any $x\in
\mathbb{R}
,$ $k\in
\mathbb{N}
$ it is easy to derive from (\ref{kth deriv}) the crude estimate 
\[
\left\vert f^{\left(k\right)}\left(x\right)\right\vert \leq\frac{L_{s}\left(\gamma\right)^{k}L_{t}\left(\sigma\right)^{k}}{k!}\exp\left(|x|L_{s}\left(\gamma\right)L_{t}\left(\sigma\right)\right),
\]
which could be refined e.g. by considering estimate on ratios of Bessel
functions $I_{k+1}/I_{k}.$ Putting things together we obtain.
\begin{prop}
Let $\mathcal{P=}\left\{ p_{n}:n\in
\mathbb{N}
\cup\left\{ 0\right\} \right\} $ be a system of orthogonal polynomials with respect to a continuous
positive weight function $w\in L^{1}\left(a,b\right).$ For every
$n$ the error in the associated quadrature is bounded above by
\begin{align*}
 & \quad\left\vert \int_{a}^{b}K^{x\gamma,\sigma}(s,t)w\left(x\right)dx-\sum_{k=0}^{n}w_{k}K^{x_{k}\gamma,\sigma}(s,t)\right\vert \\
 & \leq\frac{L_{s}\left(\gamma\right)^{2n+2}L_{t}\left(\sigma\right)^{2n+2}\exp\left(|\xi|L_{s}\left(\gamma\right)L_{t}\left(\sigma\right)\right)}{[\left(2n+2\right)!]^{2}}\int_{a}^{b}w\left(x\right)\pi_{n+1}\left(x\right)^{2}dx.
\end{align*}
\end{prop}

\begin{example}
Let $\left(a,b\right)=\left(-1,1\right),$ $w\left(x\right)=\frac{1}{\pi}\frac{1}{\sqrt{1-x^{2}}}$
as in the earlier example (\ref{acrsin}). Then $\mathcal{P}$ can
be the family of Chebyshev polynomials of the first kind $p_{n}=T_{n}$
in which case (see \cite{AS}) 
\[
\int_{a}^{b}w\left(x\right)\pi_{n+1}\left(x\right)^{2}dx=\frac{1}{2^{2n+1}}.
\]
Therefore if $\gamma$ and $\sigma$ have lengths at most $L$ the
degree $n+1$ quadrature rule results in an error at most
\[
R_{n}=\frac{L^{4n+4}\exp\left(L^{2}\right)}{2^{2n+1}[\left(2n+2\right)!]^{2}}.
\]
To give some idea of the number of points needed (and hence the number
of PDEs solutions needed), if $L=10$ then $R_{25}=e^{-8.6017}$,
$R_{30}=e^{-50.492},$ whereas if $L=100$ then $R_{1050}=e^{-49.497}$.
The ratio
\[
\frac{R_{n+1}}{R_{n}}=\frac{L^{4}}{4(2n+4)^{2}(2n+3)^{2}},
\]
articulates the trade off between the length $L$ and the number
of points $n$.

\end{example}

\begin{example}
The $\phi$-signature kernel $K_{\phi}^{\gamma,\sigma}(s,t)$
for $\phi(k)=\left(\frac{k}{2}\right)!$ is studied in Corollary
\ref{cor-sigkernel-exp}. In this case the random variable $\pi$ is exponentially
distributed, hence $\pi^{1/2}$ is Rayleigh distributed
with density $w(x)=2xe^{-x^{2}},\ x>0$. We have
\begin{equation}
K_{\phi}^{\gamma,\sigma}(s,t)=\mathbb{E}\left[K^{\pi^{1/2}\gamma,\sigma}(s,t)\right]=\int_{0}^{\infty}2K^{x\gamma,\sigma}(s,t)xe^{-x^{2}}dx.
\end{equation}
Let $f(x)=K^{x\gamma,\sigma}(s,t)$, then  
\[
K_{\phi}^{\gamma,\sigma}(s,t)=2\int_{0}^{\infty}f(x)xe^{-x^{2}}dx
\]
which can be numerically calculated by the classical Gaussian quadrature
formula (see e.g. \cite{Shizgal-1981,SBG-1969}),
\[
\int_{0}^{\infty}f(x)xe^{-x^{2}}dx\approx\sum_{k=0}^{n}w_{k}f(x_{k}).
\]
The abscissae $x_{k},\ k=0,1,\cdots,n$ are the roots of a $(n+1)$-th
degree polynomial $p_{n+1}(x)$ and $w_{k}$ are the weights of quadrature. Explicit values are given in \cite{Shizgal-1981,SBG-1969}. 
\end{example}

\section{Expected General Signature Kernels \label{sec-expected-signature-kernel}}

We develop our earlier discussion to consider how $\phi$-signature kernels can be combined with the notion of 
expected signatures to compare the laws of two stochastic processes. In the examples we study one of the measures will be
Wiener's measure, which we denote by $\mathcal{W}$ and the other will be denote by $\mu$. The measure $\mu$ will typically discrete and supported on bounded variation paths, Our aim will be to compute
\[
K_{\phi}^{\mathcal{W},\mu}\left(s,t\right)=\left\langle \mathbb{E}_{X\sim\mathcal{W}}\left[  S\left(  X\right)
_{0,1}\right]  ,\mathbb{E}_{X\sim\mu}\left[  S\left(  X\right)  _{0,1}\right]
\right\rangle _{\phi},
\]
where $ S\left(  X\right)$ denotes the Stratonovich signature of $X$. We will sometimes write $\mathbb{E}\left[S\left(\circ B\right)_{0,s}\right]$, for a Brownian motion $B$, in place of $ \mathbb{E}_{X\sim\mathcal{W}}\left[  S\left(  X\right)
_{0,1}\right] $ to emphasise the fact that the signature is constructed via Stratonvich calculus.

As an initial step, we assume that $\gamma$ is a fixed (deterministic) continuous path of  bounded variation. We look to obtain formula for the $\phi$-signature
kernel of the expected Stratonovich signature of Brownian motion and $\gamma$, i.e.
\[
K_{\phi}^{\mathcal{W},\gamma}\left(s,t\right):=\left\langle \mathbb{E}\left[S\left(\circ B\right)_{0,s}\right],S\left(\gamma\right)_{0,t}\right\rangle _{\phi}
\]
A key idea to doing this will be to use notion of the hyperbolic development
of $\gamma$ which has been used in earlier study of the signature and, in this context, was initiated by \cite{HL2010}. We summarise the essential background in the section below.

\subsection{Hyperbolic Development}

We gather the basic notation and results. Readers seeking further details can consult the references \cite{BG2019,HL2010,LX}. We let $\mathbb{H}^{d}$
denote $d$-dimensional hyperbolic space realised as the hyperboloid
$\{x\in\mathbb{R}^{d+1}:x\ast x=-1,\ x_{d+1}>0\}$ endowed with the
Minkowski product 
\[
x\ast y=\sum_{i=1}^{d}x_{i}y_{i}-x_{d+1}y_{d+1}\text{, for }x=(x_{1},\cdots,x_{d},x_{d+1})\in\mathbb{R}^{d+1}.
\]
It is well known that this defines a Riemannian metric when restricted
to the tangent bundle of $\mathbb{H}^{d}$. We let $\operatorname{d_{\mathbb{H}^{d}}}$
denote the associated Riemannian distance function and recall that
\begin{equation}
\cosh\operatorname{d_{\mathbb{H}^{d}}}(x,y)=-x\ast y,\label{eq_cosh}
\end{equation}
see e.g. \cite{Cannon}. Define the linear map $F:\mathbb{C}^{d}\rightarrow\mathcal{M}_{d+1}\left(\mathbb{C}\right)$
into the space of $d+1$ by $d+1$ matrices over $\mathbb{C}$ by
\begin{equation}
F:x\rightarrow\begin{pmatrix}0 & x\\
x^{T} & 0
\end{pmatrix}.\label{eq-F-map}
\end{equation}
Then if $V$ is a real inner product space of dimension $d$ and $\gamma:[a,b]\rightarrow V$
is continuous path of bounded variation then, by fixing an orthonormal basis of
$V,$ and writing $\gamma$ in this basis as $\gamma=(\gamma_{t_{1}},...,\gamma_{t_{d}})$ we can solve the linear differential equation 
\begin{equation}
d\Gamma_{s,t}(u)=F(d\gamma(u))\Gamma_{s,t}(u),\ u\in\lbrack s,t]\subset\lbrack a,b],\text{with }\Gamma_{s,t}(s)=I=I_{d+1}\label{eq_Gamma_ode}
\end{equation}
uniquely. In the case the map $\gamma|_{[s,t]}\mapsto\Gamma_{s,t}\left(\cdot\right)$
takes a path segment in $V$ into one in the isometry group of $\mathbb{H}^{d}.$ The resulting $\Gamma_{s,t}\left(\cdot\right)~$ is called the \textit{Cartan Development
of the path segment} $\gamma|_{[s,t]}.$ It satisfies the multiplicative
property 
\begin{equation}
\Gamma_{u,t}(t)\Gamma_{s,u}(u)=\Gamma_{s,t}(t),\ s\leq u\leq t.\label{eq_multiplicative}
\end{equation}
To simplify things we write $\Gamma(t):=\Gamma^{\gamma}(t):=\Gamma_{a,b}(t)$
for $t\in\lbrack a,b]$. It is elementary to represent $\Gamma$ as
the convergent series 
\begin{equation}
\Gamma(t)=I+\sum_{n=1}^{\infty}\int_{a<t_{1}<\cdots<t_{n}<t}F(d\gamma({t_{1}}))\cdots F(d\gamma({t_{n}})).\label{eq_Gamma_sig}
\end{equation}
Then letting $o=(0,\cdots,0,1)^{T}\in\mathbb{H}^{d}, $we define $\sigma(t):=\Gamma(t)o$
to be the \textit{hyperbolic development of the path} $\gamma$ onto
$\mathbb{H}^{d}$, and we write $\sigma_{\gamma}$ to emphasise the dependence
on $\gamma.$

A global coordinate chart for $\mathbb{H}^{d}$ is determined
by $\mathbb{H}^{d}\ni m\mapsto\left(\eta,\rho\right)\in\mathbb{S}^{d-1}\times$
$\mathbb{R}_{+}$ where $\left(\eta\sinh\rho,\cosh\rho\right)=m.$
Using these coordinates, we define 
\[
\eta\left(t\right)=\eta_{\gamma}(t)=\eta\left(\sigma_{\gamma}\left(t\right)\right)\in\mathbb{S}^{d-1}\text{ and }\rho\left(t\right)=\rho_{\gamma}\left(t\right)=\rho\left(\sigma_{\gamma}\left(t\right)\right)\in\mathbb{R}_{+}.
\]
The following identity follows from (\ref{eq_Gamma_sig}) and (\ref{eq_cosh}):
\begin{equation}
\cosh\rho_{\gamma}(t)=\Gamma_{d+1,d+1}(t)=1+\sum_{n=1}^{\infty}\int_{a<t_{1}<\cdots<t_{2n}<t}\langle d\gamma({t_{1}}),d\gamma({t_{2}})\rangle\cdots\langle d\gamma({t_{2n-1}}),d\gamma({t_{2n}})\rangle,%
\label{eq_cosh_series}
\end{equation}
where $\Gamma\left(t\right)=\left(\Gamma_{ij}\left(t\right)\right)_{i,j=1,\cdots,d+1}$.
We will need to broaden this discussion to consider the development of paths after complex rescaling. To this end, if $\gamma$ is as above and $z \in \mathbb{C}$ then we let $z\gamma$ denote the path in  $V^{\mathbb{C}}$,  the complexification of $V$. We will be interested in the relationship
between the solution to (\ref{eq_Gamma_ode}), when $\gamma$ is replaced by $z\gamma$, and the series (\ref{eq_cosh_series}).
The following lemma identifies the structure we need.
\begin{lem}
\label{lem-complex-hd}Let $\gamma:[a,b]\rightarrow V$ be a continuous
path of bounded variation. For $z\in
\mathbb{C}
$ let $z\gamma:\left[a,b\right]\rightarrow V^{\mathbb{C}}$ be the rescaling
of $\gamma$ by $z\in \mathbb{C}$. Given an orthonormal basis of $V$, write $\gamma_{t}=$ $\left(\gamma_{t}^{1},...,\gamma_{t}^{d}\right) \in \mathbb{R}^{d}$ and $z\gamma\left(t\right):=\left(z\gamma_{t}^{1},...,z\gamma_{t}^{d}\right) \in \mathbb{C}^{d}$
in terms of this basis. Then 
\begin{equation}
d\Gamma^{z\gamma}(u)=F(d\left(z\gamma\right)(u))\Gamma^{z\gamma}(u),\ u\in\lbrack a,b],\text{with }\Gamma^{z\gamma}(s)=I_{d+1}\label{lin}
\end{equation}
has a unique solution in $\mathcal{M}_{d+1}\left(\mathbb{C}\right)$
and furthermore the entry 
\begin{equation}
\Gamma_{d+1,d+1}^{z\gamma}(t)=1+\sum_{n=1}^{\infty}z^{2n}\int_{0<t_{1}<...<t_{2n}<t}\left\langle d\gamma_{t_{1}},d\gamma_{t_{2}}\right\rangle ...\left\langle d\gamma_{t_{2n-1}},d\gamma_{t_{2n}}\right\rangle .\label{series-1}
\end{equation}
If $\gamma$ is a piecewise linear path defined by the concatenation
\[
\gamma_{v_{1}}\ast\gamma_{v_{2}}....\ast\gamma_{v_{n}}:\left[a,b\right]\rightarrow V,
\]
i.e. $\gamma$ is such that $\gamma_{v_{i}}^{\prime}\left(t\right)=v_{i}\in \mathbb{R}^{d}$
for $t\in(t_{i-1},t_{i})$. Then the solution to (\ref{lin}) is given
explicitly by the matrix product
\begin{equation}
\Gamma^{z\gamma}\left(b\right)=A\left(v_{n},\Delta_{n},z\right)A\left(v_{n-1},\Delta_{n-1},z\right)\cdots A\left(v_{1},\Delta_{1},z\right),\label{eq-piecewise-solution}
\end{equation}
where $\Delta_{i}=t_{t}-t_{i-1}$ and 
\begin{equation}
A\left(v,\Delta,z\right):=I_{d+1}+\sinh\left(z\left\vert v\right\vert \Delta\right)M+\left(\cosh\left(z \left\vert v\right\vert \Delta\right)-1\right)M^{2}\label{eq-piecewise-onepiece}
\end{equation}
in which 
\[
M=\begin{pmatrix}0 & \tilde{v}\\
\tilde{v}^{T} & 0
\end{pmatrix}\in\mathcal{M}_{d+1}\left(\mathbb{R}\right)\text{ with }\tilde{v}=\frac{v}{\left\vert v\right\vert }.
\]
\end{lem}

\begin{proof}
Since the ODE (\ref{lin}) is linear, there is a unique solution $\Gamma^{z\gamma}(t)$
which can be represented by equation (\ref{eq_Gamma_sig}) by replacing
$\gamma$ with $z\gamma$. Then equation (\ref{series-1}) can be
obtained by taking the last entry of this equation. 

To obtain the explicit solution in the case where $\gamma$ is piecewise linear path, we first assume $\gamma'=v$ 
on $[s,t]$. Then by using the observation that $M^3=M$ together with equation (\ref{eq_Gamma_sig}), we have 
\[
\begin{aligned}\Gamma_{s,t}^{z\gamma}(t) & =I+\sum_{n=1}^{\infty}\frac{(z|v|)^{2n-1}(t-s)^{2n-1}}{(2n-1)!}M+\sum_{n=1}^{\infty}\frac{(z|v|)^{2n}(t-s)^{2n}}{(2n)!}M^{2}\\
 & =I+\sinh\left(z\left\vert v\right\vert (t-s)\right)M+\left(\cosh\left(z\left\vert v\right\vert (t-s)\right)-1\right)M^{2}.
\end{aligned}
\]
In the general case, the multiplicative property (\ref{eq_multiplicative}) together with simple induction argument implies that the solution has the form (\ref{eq-piecewise-solution}). 
\end{proof}

\subsection{Signature Kernels and Hyperbolic Development}

We begin this subsection by giving a closed form of the $\phi$-signature
kernel $K_{\phi}^{\mathcal{W},\mu}\left(s,t\right)$ for the special case
$\phi(k)=\left(\frac{k}{2}\right)!$ based on the theory presented above. 
\begin{thm}
\label{hyperbolic}\textbf{(Formula for $\left\langle \mathbb{E}\left[S\left(\circ B\right)\right],S\left(\gamma\right)\right\rangle _{\phi}$)}
Let $\phi:
\mathbb{N}
\cup\left\{ 0\right\} \rightarrow
\mathbb{R}
_{+}$ be defined by $\phi(k)=\left(\frac{k}{2}\right)!$ for $k\in \mathbb{N}\cup\{0\}$. Suppose that
$B$ is a $d$-dimensional Brownian motion, then the expected Stratonovich
signature, $\mathbb{E}\left[S\left(\circ B\right)_{0,s}\right]$,
belongs to $T_{\phi}\left(V\right)$ for any $0\leq s<\infty.$ Furthermore
if $\gamma:[0,1]\to V$ is any continuous path of bounded variation
it holds that
\begin{equation}
K_{\phi}^{\mathcal{W},\gamma}\left(s,t\right):=\left\langle \mathbb{E}\left[S\left(\circ B\right)_{0,s}\right],S\left(\gamma\right)_{0,t}\right\rangle _{\phi}=\cosh\left(\rho_{\sqrt{s/2}\gamma}\left(t\right)\right).\label{identity}
\end{equation}
In this notation $\rho_{\lambda}\left(t\right):=\operatorname{d_{\mathbb{H}^{d}}}\left(o,\sigma_{\lambda\gamma}\left(t\right)\right)$
is the distance between the hyperbolic development $\sigma_{\lambda\gamma}(t)$
of the path $\lambda\gamma\left(\cdot\right)$ from $T_{o}\mathbb{H}^{d}$
onto the $d$-dimensional hyperbolic space $\mathbb{H}^{d}$ started at the
base point $o\in\mathbb{H}^{d}$, and $\operatorname{d_{\mathbb{H}^{d}}}:\mathbb{H}^{d}\times\mathbb{H}^{d}\rightarrow\lbrack0,\infty)$
is the Riemannian distance on $\mathbb{H}^{d}$. 
\end{thm}

\begin{proof}
For the first assertion recall that (see e.g. Proposition 4.10. in
\cite{LV-2004})
\[
\mathbb{E}\left[S\left(\circ B\right)_{0,s}\right]=\exp\left(\frac{s}{2}\sum_{i=1}^{d}e_{i}^{2}\right)=\sum_{k=0}^{\infty}\frac{s^{k}}{2^{k}k!}\sum_{i_{1},...,i_{k}=1}^{d}e_{i_{1}}^{2}...e_{i_{k}}^{2}
\]
so that 
\[
\left\vert \left\vert \mathbb{E}\left[S\left(\circ B\right)_{0,s}\right]\right\vert \right\vert _{\phi}^{2}=\sum_{k=0}^{\infty}k!\frac{s^{2k}d^{k}}{2^{2k}(k!)^{2}}=e^{s^{2}d/4}<\infty.
\]
For the second assertion we have that 
\[
\left\langle \mathbb{E}\left[S\left(\circ B\right)_{0,s}\right],S\left(\gamma\right)_{0,t}\right\rangle _{\phi}=\sum_{k=0}^{\infty}\frac{s^{k}}{2^{k}}\int_{0<t_{1}<...<t_{2k}<t}\left\langle d\gamma_{t_{1}},d\gamma_{t_{2}}\right\rangle ...\left\langle d\gamma_{t_{2k-1}},d\gamma_{t_{2k}}\right\rangle .
\]
The right hand side of this expression equals that of (\ref{identity});
see formula (\ref{eq_cosh_series}). 
\end{proof}
In the following, we give some remarks on the computation of this
basic signature kernel based on the above theorem. 
\begin{rem}
(1) In contrast to the earlier case of two paths, we need only solve an ODE to calculate $\left\langle \mathbb{E}\left[S\left(\circ B\right)\right],S\left(\gamma\right)\right\rangle _{\phi}$
and not a PDE. (2) For general $\gamma$,
the ODE is known, and is determined by the linear vector fields in equation
(\ref{eq_Gamma_ode}). Any ODE solver such as Runge-Kutta could in
principle be used to obtain numerical solutions. (3) For piecewise linear case, the exact solution
is given in equation (\ref{eq-piecewise-solution}) as a product of
matrices. 
\end{rem}

\subsection{The Original Kernel for Expected Signatures}

Theorem \ref{hyperbolic} gives a closed form expression for the $\phi$-signature
kernel of Stratonovich expected signature of Brownian motion and the
signature of a bounded variation continuous path where $\phi(k)=\left(\frac{k}{2}\right)!$.
As previously we will be interested in related formulae for different signature kernels.
We can obtain these formulae by using an extension of the ideas developed earlier
in the paper. In the case of the original signature kernel (i.e. $\phi\equiv1$),
we can make use of the classical integral representation of the reciprocal
gamma function which for integers has the form:
\begin{equation}
\frac{1}{k!}=\frac{1}{2\pi i}\oint_{C}z^{-(k+1)}e^{z}dz=\frac{1}{2\pi}\int_{-\pi}^{\pi}e^{-ik\theta}e^{e^{i\theta}}d\theta \label{eq-rf}
\end{equation}
where $\oint_{C}$ denotes the contour integral around the unit circle traversed once anticlockwise. This is an instance of the more general formula
\begin{equation}
\frac{1}{\Gamma(p)}=\frac{1}{2\pi i}\oint_{H}z^{-p}e^{z}dz, \label{eq-rf2}
\end{equation}
where $H$ is Hankel contour which winds from $-\infty-0i$
in the lower half-plane, anticlockwise around 0, and then back to $-\infty+0i$ in the
upper half-plane, while respecting the branch cut of the integrand along the negative real axis. The advantage of using these integral representation is twofold. 
First, the integrand has exponential dependence on $k$ making it suitable to employ the 
techniques developed earlier in the paper. Second the underlying numerical integration theory is well developed and the convergence rates for optimised quadrature formulae are
exceedingly fast. We give some examples below but refer the reader to the reference
\cite{TWS-2006} for further details. 
We have the following theorem. 
\begin{thm}
\label{thm-kernel-BM-original}Let $\phi\equiv1$. Suppose $B$ is
a $d$-dimensional Brownian motion, then the expected Stratonovich
signature, $\mathbb{E}\left[S\left(\circ B\right)_{0,s}\right]$,
belongs to $T_{\phi}\left(V\right)$ for any $0\leq s<\infty$ and
\begin{equation}
\left\vert \left\vert \mathbb{E}\left[S\left(\circ B\right)_{0,s}\right]\right\vert \right\vert _{\phi}^{2}=\frac{1}{2\pi i}\oint_{C}z^{-1}e^{z+s^{2}d/(4z)}dz\label{eq-norm-BM-original}
\end{equation}
where the contour $C$ is the unit circle in $\mathbb{C}$ traversed anticlockwise. Furthermore if
$\gamma$ is any continuous path of bounded variation it holds that
\begin{equation}
K_{\phi}^{\mathcal{W},\gamma}\left(s,t\right):=\left\langle \mathbb{E}\left[S\left(\circ B\right)_{0,s}\right],S\left(\gamma\right)_{0,t}\right\rangle _{\phi}=\frac{1}{2\pi i}\oint_{C}z^{-1}e^{z}\Gamma_{d+1,d+1}^{c_{s}(z)\gamma}(t)dz\label{eq-kernel-BM-original}
\end{equation}
where $c_{s}(z)=\sqrt{s/2z}\in\mathbb{C}$ and $\Gamma_{d+1,d+1}^{c_{s}(z)\gamma}(t)$
is defined by the series (\ref{series-1}), i.e. the last entry of the solution
to ODE (\ref{lin}).
\end{thm}

\begin{proof}
Using the definition of the original signature kernel and the dominated convergence theorem to interchange the order of $\sum$ and $\oint_{C}$ we have
\[
\begin{aligned}\left\langle \mathbb{E}\left[S\left(\circ B\right)_{0,s}\right],S\left(\gamma\right)_{0,t}\right\rangle _{\phi} & =\sum_{k=0}^{\infty}\frac{1}{k!}\frac{s^{k}}{2^{k}}\int_{0<t_{1}<...<t_{2k}<t}\left\langle d\gamma_{t_{1}},d\gamma_{t_{2}}\right\rangle ...\left\langle d\gamma_{t_{2k-1}},d\gamma_{t_{2k}}\right\rangle \\
 & =\frac{1}{2\pi i}\oint_{C}z^{-1}e^{z}\left(\sum_{k=0}^{\infty}z^{-k}\frac{s^{k}}{2^{k}}\int_{0<t_{1}<...<t_{2k}<t}\left\langle d\gamma_{t_{1}},d\gamma_{t_{2}}\right\rangle ...\left\langle d\gamma_{t_{2k-1}},d\gamma_{t_{2k}}\right\rangle \right)dz.
\end{aligned}
\]
If $c_{s}(z)=\sqrt{s/2z}$ then by equation (\ref{series-1}),
we know that 
\[
\sum_{k=0}^{\infty}z^{-k}\frac{s^{k}}{2^{k}}\int_{0<t_{1}<...<t_{2k}<t}\left\langle d\gamma_{t_{1}},d\gamma_{t_{2}}\right\rangle ...\left\langle d\gamma_{t_{2k-1}},d\gamma_{t_{2k}}\right\rangle =\Gamma_{d+1,d+1}^{c_{s}(z)\gamma}(t),
\]
which is the last entry of the solution $\Gamma^{c_{s,z}\gamma}(t)$
to ODE (\ref{lin}). The argument for the squared norm
of Brownian motion, follows a similar pattern and yields
\[
\left\vert \left\vert \mathbb{E}\left[S\left(\circ B\right)_{0,s}\right]\right\vert \right\vert _{\phi}^{2}=\frac{1}{2\pi i}\oint_{C}z^{-1}e^{z}\left(\sum_{k=0}^{\infty}z^{-k}\frac{s^{2k}d^{k}}{2^{2k}k!}\right)dz=\frac{1}{2\pi i}\oint_{C}z^{-1}e^{z}e^{s^{2}d/(4z)}dz.
\]

\end{proof}

\subsubsection*{Computation of the contour integrals}

The implementation of the formula above demands an efficient way to approximate contour integrals of the form 
\begin{equation}
I=\frac{1}{2\pi i}\oint_{C}e^{z}f(z)dz=\frac{1}{2\pi}\int_{-\pi}^{\pi}e^{e^{i\theta}}f(e^{i\theta})e^{i\theta} d\theta.
\end{equation}
A natural approach is to apply a trapezoidal rule based on $N$ equally spaced points on the unit circle, i.e. to approximate $I$ using 
\begin{equation}
I_{N}= \frac{1}{N} \sum_{k=1}^{N} e^{z_{k}}f(z_{k})z_{k}
\end{equation}
where $z_{k}=e^{2k\pi i/N}$. 
Several other methods have been proposed in Trefethen, Weideman
and Schmelzer (2006) \cite{TWS-2006} for the efficient approximation of
the Hankel-type contour integrals of the form
\[
I=\frac{1}{2\pi i}\oint_{H}e^{z}f(z)dz.
\]
The idea is to seek an optimal selection of contour according to the number of points in the quadrature formula. Letting $\varphi(\theta)$ be an analytic function that maps the real
line $\mathbb{R}$ onto the contour $H$. Then the approach is to approximate
\[
I=\frac{1}{2\pi i}\int_{-\infty}^{+\infty}e^{\varphi(\theta)}f(\varphi(\theta))\varphi'(\theta)d\theta
\]
by 
\begin{equation}
I_{N}=-iN^{-1}\sum_{k=1}^{N}e^{z_{k}}f(z_{k})w_{k}=-\sum_{k=1}^{N}c_{k}f(z_{k})\label{eq-contour-int-appr}
\end{equation}
on the finite interval $[-\pi,\pi]$ with $N$ points which are regularly
spaced on the interval and $z_{k}=\varphi(\theta_{k})$, $w_{k}=\varphi'(\theta_{k})$
and $c_{k}=iN^{-1}e^{z_{k}}w_{k}$. The convergence rates for these
optimised quadrature formulae are very fast, of order $O(3^{-N})$.
Three classes of contours have been investigated in \cite{TWS-2006}:
\begin{itemize}
\item Parabolic contours
\[
\varphi(\theta)=N(0.1309-0.1194\theta^{2}+0.2500i\theta)
\]
\item Hyperbolic contours
\[
\varphi(\theta)=2.246N(1-\sin(1.1721-0.3443i\theta))
\]
\item Cotangent contours
\[
\varphi(\theta)=N(0.5017\theta\cot(0.6407\theta)-0.6122+0.2645i\theta)
\]
\end{itemize}
Note in each case the dependence of the family on $N$.

The procedure for computing the kernel in equation (\ref{eq-kernel-BM-original})
is first compute the function $\Gamma_{d+1,d+1}^{c_{s}(z)\gamma}(t)$
by utilising the explicit formula (\ref{eq-piecewise-solution}) for
piecewise linear paths. By taking 
\[
f(z)=z^{-1}\Gamma_{d+1,d+1}^{c_{s}(z)\gamma}(t)
\]
we can approximate the contour integral by one of the approaches described above.

\subsection{Expected signatures for general kernels}

The representation of the previous subsection can be combined with
the ideas of Section \ref{sec-Signature-Kernels} to obtain similar
representations for $\left\langle \mathbb{E}\left[S\left(\circ B\right)\right],S\left(\gamma\right)\right\rangle _{\phi}$
for general $\phi$ satisfying the conditions of Theorem \ref{general prop}.
The expression is as follows. 

\begin{thm}
Let $\mu$ be a finite signed Borel measure $\mu$ on $\mathbb{R}$.
Suppose that $\phi:\mathbb{N}\cup\{0\}\to\mathbb{C}$ is such that
\[
\phi\left(k\right)=\int_{G}r\left(k,\widetilde{z}\right)\mu\left(d\widetilde{z}\right)\in
\mathbb{C}
\,\text{, for all }k\in\mathbb{N}\cup\{0\}
\]
where $r\left(k,\cdot\right)$ is assumed to have the form $r\left(k,\widetilde{z}\right)=$
$g\left(\widetilde{z}\right)^{\alpha k}$ $\in
\mathbb{C}
$ for $\alpha\in
\mathbb{R}
$ and some function $g:\mathbb{C}\rightarrow
\mathbb{C}
.$ We assume that $\phi$ satisfies the conditions in Theorem \ref{general prop},
and that $B$ a $d$-dimensional Brownian motion. Then the expected Stratonovich
signature, $\mathbb{E}\left[S\left(\circ B\right)_{0,s}\right]$,
belongs to $T_{|\phi|}\left(V\right)$ for any $0\leq s<\infty$ and
\begin{equation}
\left\vert \left\vert \mathbb{E}\left[S\left(\circ B\right)_{0,s}\right]\right\vert \right\vert _{\phi}^{2}=\frac{1}{2\pi i}\oint_{C}\int_{G}\left[z^{-1}e^{z}\exp\left(\frac{g(\widetilde{z})^{2\alpha}s^{2}d}{4z}\right)\right]\mu(d\widetilde{z})dz
\end{equation}
where $C$ is unit circle in $\mathbb{C}$ traversed anticlockwise. Furthermore if
$\gamma$ is any continuous path of bounded variation it holds that
\begin{equation}
K_{\phi}^{\mathcal{W},\gamma}\left(s,t\right):=\left\langle \mathbb{E}\left[S\left(\circ B\right)_{0,s}\right],S\left(\gamma\right)_{0,t}\right\rangle _{\phi}=\frac{1}{2\pi i}\oint_{C}\int_{G}\left[z^{-1}e^{z}\Gamma_{d+1,d+1}^{c_{g,\alpha,s}(\widetilde{z},z)\gamma}(t)\right]\mu(d\widetilde{z})dz\label{eq-bilinear-exp}
\end{equation}
where $c_{g,\alpha,s}(\widetilde{z},z):=g(\widetilde{z})^{\alpha}\sqrt{s/(2z)}\in\mathbb{C}$
and $\Gamma_{d+1,d+1}^{c_{g,\alpha,s}(\widetilde{z},z)\gamma}(t)$
is the series (\ref{series-1}), i.e. the last entry of the solution
to ODE (\ref{lin}).
\end{thm}

\begin{proof}
The conditions for $\phi$ in Theorem \ref{general prop} and by now standards estimates allow for the steps of the proof of Theorem \ref{thm-kernel-BM-original} to be repeated making the obvious modifications.
\end{proof}
As a special case, if $\phi$ is the moments of a
random variable $\pi$, i.e. 
\begin{equation}
\phi(k)=\mathbb{E}[\pi^{k}],\ \forall k\geq0,\label{eq-moments-rv}
\end{equation}
the representations are as follows.
\begin{cor}
\label{cor-kernel-BM-rv} Let the function $\phi:
\mathbb{N}
\cup\left\{ 0\right\} \rightarrow
\mathbb{R}
$ as defined in (\ref{eq-moments-rv}) and $\psi(k)=\mathbb{E}[|\pi|^{k}]$
such that $\psi$ satisfies Condition \ref{sum}. Suppose $B$ is
a $d$-dimensional Brownian motion, then the expected Stratonovich
signature, $\mathbb{E}\left[S\left(\circ B\right)_{0,s}\right]$,
belongs to $T_{|\phi|}\left(V\right)$ for any $0\leq s<\infty$ and
\begin{equation}
\left\vert \left\vert \mathbb{E}\left[S\left(\circ B\right)_{0,s}\right]\right\vert \right\vert _{\phi}^{2}=\frac{1}{2\pi i}\oint_{C}z^{-1}e^{z}\mathbb{E}_{\pi}\left[e^{(\pi s)^{2}d/(4z)}\right]dz.
\end{equation}
If $\gamma$ is
any continuous path of bounded variation it holds that
\begin{equation}
K_{\phi}^{\mathcal{W},\gamma}\left(s,t\right)\left(s,t\right):=\left\langle \mathbb{E}\left[S\left(\circ B\right)_{0,s}\right],S\left(\gamma\right)_{0,t}\right\rangle _{\phi}=\frac{1}{2\pi i}\oint_{C}z^{-1}e^{z}\mathbb{E}_{\pi}\left[\Gamma_{d+1,d+1}^{c_{s}(\pi,z)\gamma}(t)\right]dz\label{eq-kernel-BM-rv}
\end{equation}
where $c_{s}(x,z):=x\sqrt{s/(2z)}\in\mathbb{C}$ and $\Gamma_{d+1,d+1}^{c_{s}(x,z)\gamma}(t)$
is the series (\ref{series-1}).
\end{cor}

As an example, we recall the case $\phi(k)=\frac{\Gamma(m+1)\Gamma(k+1)}{\Gamma(k+m+1)}$
studied already in Section \ref{sec-Signature-Kernels}. Suppose the
random variable $\pi\sim\textrm{Beta}(1,m)$ is Beta distributed,
then the moments of $\pi$ are
\[
\mathbb{E}[\pi^{k}]=\frac{B(k+1,m)}{B(1,m)}=\phi(k).
\]
 We then have the following.
\begin{example}
\label{eg-kernel-BM-Beta}Let $\phi(k)=\frac{\Gamma(m+1)\Gamma(k+1)}{\Gamma(k+m+1)}$
and $B$ a $d$-dimensional Brownian motion. Then $\phi$ satisfies
Condition \ref{sum}. The expected Stratonovich signature, $\mathbb{E}\left[S\left(\circ B\right)_{0,s}\right]$,
is well defined and belongs to $T_{\phi}\left(V\right)$ for any $0\leq s<\infty$,
and the squared norm 
\begin{equation}
\left\vert \left\vert \mathbb{E}\left[S\left(\circ B\right)_{0,s}\right]\right\vert \right\vert _{\phi}^{2}=\frac{\Gamma(m+1)}{2\pi i}\oint_{C}z^{-(m+1)}e^{z}\frac{dz}{\sqrt{1-s^{2}d/z^{2}}}\label{eq-sqnorm-Beta}
\end{equation}
If $\gamma$ is
any continuous path of bounded variation, then
\begin{equation}
K_{\phi}^{\gamma,W}\left(s,t\right)=\frac{\Gamma(m+1)}{2\pi i}\oint_{C}z^{-(m+1)}e^{z}\left[\frac{1}{\sqrt{2\pi}}\int_{-\infty}^{+\infty}\Gamma_{d+1,d+1}^{c_{s}(x,z)\gamma}(t)e^{-\frac{x^{2}}{2}}dx\right]dz\label{eq-kernel-exp-Beta}
\end{equation}
where $c_{s}(x,z)=z^{-1}x\sqrt{s}\in\mathbb{C}$ and $\Gamma_{d+1,d+1}^{c_{s}(x,z)\gamma}(t)$
is the series (\ref{series-1}).
\end{example}

The representations above are slightly different from Corollary \ref{cor-kernel-BM-rv}
in which $\pi$ should be a Beta random variable. The expressions
above are obtained by the formulas below: 
\[
\frac{\Gamma(2k+1)}{2^{k}k!}=(2k-1)!!=\mathbb{E}_{X}[X^{2k}]\ \text{and}\ \frac{1}{\Gamma(2k+m+1)}=\oint_{C}z^{-(2k+m+1)}e^{z}dz
\]
where $X\sim N(0,1)$ is a standard normal random variable. In the
point view of computation, the Gaussian quadrature for approximating
the formula (\ref{eq-kernel-exp-Beta}) is much easier than using
the formula (\ref{eq-kernel-BM-rv}) with $\pi\sim\textrm{Beta}(1,m)$. 
\begin{rem}
In terms of the computation procedure, we take the signature kernel
in equation (\ref{eq-kernel-exp-Beta}) as an example. It can be calculated
in three successive steps. First, for fixed $z$, $x$ and $s$, get
the exact value of $\Gamma_{d+1,d+1}^{c_{s}(x,z)\gamma}(t)$ by the
explicit solution (\ref{eq-piecewise-solution}) to ODE (\ref{lin})
for piecewise linear path. Second, approximate the expectation 
\[
\mathbb{E}_{X}\left[\Gamma_{d+1,d+1}^{c_{s}(X,z)\gamma}(t)\right]=\frac{1}{\sqrt{2\pi}}\int_{-\infty}^{+\infty}\Gamma_{d+1,d+1}^{c_{s}(x,z)\gamma}(t)e^{-\frac{x^{2}}{2}}dx
\]
by classical Gaussian quadrature on the whole real line.
Third, approximate the contour integral using one of the methods described above.
The steps are summarised schematically as
follows:
\[
K_{\phi}^{\mathcal{W},\gamma}\left(s,t\right)=\frac{\Gamma(m+1)}{2\pi i}\underbrace{\oint_{C}z^{-(m+1)}e^{z}}_{(3)\ Contour\ approximation}\left[\underbrace{\frac{1}{\sqrt{2\pi}}\int_{-\infty}^{+\infty}\overbrace{\Gamma_{d+1,d+1}^{c_{s}(x,z)\gamma}(t)}^{(1)\ explicit\ solution}e^{-\frac{x^{2}}{2}}dx}_{(2)\ Gaussian\ quadrature}\right]dz.
\]
The general form (\ref{eq-bilinear-exp}) can also be computed by
these three steps successively but the quadrature formula will generally be 
more complicated to implement than the Beta random variable case. See Section \ref{subsec-Quadrature-Error-Estimates}
for details. 
\end{rem}

\section{Optimal Discrete Measures on Paths\label{sec-Optimal-Discrete-Measure}%
}

In the previous sections, we have introduced the $\phi$-signature kernels. We
described method for the evaluation of these kernels for a pair of continuous
bounded variation paths, and derived a closed-form expression for the expected
signature against Brownian motion. In particular, given a finite collection of
continuous bounded variation paths $\{\gamma_{1},\gamma_{2},\cdots,\gamma
_{n}\}$ on $V$ and a discrete measure  $\mu=\sum_{i=1}^{n}\lambda_{i}%
\delta_{\gamma_{i}}$supported on this set we can evaluate%
\[
\left\vert \left\vert \mathbb{E}_{X\sim\mu}\left[  S\left(  X\right)
_{0,1}\right]  \right\vert \right\vert _{\phi}^{2}=\sum_{i,j=1}^{n}%
\lambda_{i}\lambda_{j}K_{\phi}^{\gamma_{i},\gamma_{j}},
\]
and also
\[
\left\langle \mathbb{E}_{X\sim\mathcal{W}}\left[  S\left(  X\right)
_{0,1}\right]  ,\mathbb{E}_{X\sim\mu}\left[  S\left(  X\right)  _{0,1}\right]
\right\rangle _{\phi},
\]
where $\mathcal{W}$ denotes the Wiener measure. This can be used to measure
the similarity of  using the maximum mean discrepancy distance associated with
the $\phi$signature kernel:%
\[
d_{\phi}^{2}\left(  \mathcal{W},\mu\right)  =\left\vert \left\vert
\mathbb{E}_{X\sim\mathcal{W}}\left[  S\left(  X\right)  _{0,1}\right]
-\mathbb{E}_{X\sim\mu}\left[  S\left(  X\right)  _{0,1}\right]  \right\vert
\right\vert _{\phi}^{2},
\]
which can be used as the basis of goodness-of-fit tests to measure the
similarity of $\mu$ to Wiener measure. We refer to \cite{Gretton-2012} and
\cite{CO-2018}  where kernels have been proposed as a way to support similar
analyses. 

Changing our perspective, we can also attempt to find the optimiser over some
subset of measures $C,$.i.e.
\begin{equation}
\mu^{\ast}=\arg\min_{\mu\in C}\left\vert \left\vert \mathbb{E}_{X\sim
\mathcal{W}}\left[  S\left(  X\right)  _{0,1}\right]  -\mathbb{E}_{X\sim\mu
}\left[  S\left(  X\right)  _{0,1}\right]  \right\vert \right\vert _{\phi}%
^{2}\label{eq_minimise_mu}%
\end{equation}
to give the $d_{\phi}$-best approximation to Wiener measure on $C$. An example
in which this is tractable is when the support of $\mu$ in $C$ is fixed to be
$\{\gamma_{1},\gamma_{2},\cdots,\gamma_{n}\}$ and where the set over which the
optimisation is carried our is the set of probability measures with this
support. In other words, $C$ can be identified with the simplex $C_{n}%
=\left\{  \lambda:\sum_{i=}^{n}\lambda_{i}=1,\lambda_{i}\geq0\right\}  $. By
finding this optimum we can then compare the value $d_{\phi}\left(
\mathcal{W},\mu\right)  $, for a given measure $\mu,$ to the optimised value
$d_{\phi}\left(  \mathcal{W},\mu^{\ast}\right)  $ to and use as a guide to
whether $\mu$ is $d_{\phi}$-close to $\mathcal{W}$ when compared to discrete
measures having the same support. A closely related, although more advanced
problem, is the $\phi-$cubature problem of solving
\[
(\mu^{\ast},\{\gamma_{i}\}^{\ast})=\arg\min_{(\mu,\{\gamma_{i}\})}\left\vert
\left\vert \mathbb{E}\left[  S\left(  \circ B\right)  _{0,1}\right]
-\sum_{i=1}^{n}\lambda_{i}S\left(  \gamma_{i}\right)  _{0,1}\right\vert
\right\vert _{\phi}^{2},
\]
which in the case where $\phi\left(  n\right)  =0$ for $n\geq N$ corresponds
to find a degree$-N$ cubature formula in the sense of  \cite{LV-2004}. For $N$
large enough this can be minimised (not necessarily uniquely) to zero and
explicit formulas for $\left(  \lambda_{i},\gamma_{i}\right)  $ are known in
some case; again see \cite{LV-2004} for more details

\subsection{Existence and Uniqueness of Optimal Discrete Measure}

In this subsection, we consider in detail the problem described above. We
give conditions on the collection $\{\gamma_{1},\gamma_{2},\cdots,\gamma
_{n}\}$ so that
\[
L\left(  \mu\right)  =d_{\phi}^{2}\left(  \mathcal{W},\mu\right)
\]
has a unique minimiser on the set
\[
C_{n}=\left\{  \mu=%
{\textstyle\sum\nolimits_{i=1}^{n}}
\lambda_{i}\delta_{\gamma_{i}}:\lambda_{i}\geq0,\lambda_{1}+...+\lambda
_{n}=1\right\}  .
\]
In order to find the optimal discrete measure on the set of paths
$\{\gamma_{i}\}_{i=1}^{n}$, we could solve the problem in equation
(\ref{eq_minimise_mu}) with constraints $\lambda_{i}\geq0$ and $\sum_{i=1}%
^{n}\lambda_{i}=1$. This is equivalent to solving the quadratic optimisation
problem of quadratic functions with linear equality and inequality constraints
given by
\begin{equation}%
\begin{split}
&  \min_{x\in\mathbb{R}^{n}}\frac{1}{2}x^{T}Kx-h^{T}x\\
&  \text{subject to}\ \mathbf{1}^{T}x=1,\ x\geq0.
\end{split}
\label{eq_QP}%
\end{equation}
where
\[
K=\left(  K_{\phi}^{\gamma_{i},\gamma_{j}}\right)  _{i,j=1,\cdots,n},\text{
and }h=\left(  K_{\phi}^{\gamma_{1},\mathcal{W}},\cdots,K_{\phi}^{\gamma
_{n},\mathcal{W}}\right)  ^{T}.
\]
Existence and uniqueness of the optimal solution is guaranteed by the positive
definiteness of $K$. Some sufficient conditions for positive definiteness can
be obtained from the following lemma.

\begin{lem}
The set of all signatures $\mathcal{S}$ of continuous bounded variation
paths is a linearly independent subset of $T\left( \left( V\right) \right).$
\end{lem}

\begin{proof}
Suppose that $\left\{  h_{1},...,h_{n}\right\}  $ is a subset of $\mathcal{S}$
and suppose that $\sum_{i=1}^{n}\lambda_{i}h_{i}=0$ with not all $\lambda
_{i}=0,$ e.g., suppose that $\lambda_{j}\neq0.$ The vectors $h_{1},...h_{n}$
are distinct and so there exist linear functionals $f_{i}$ on $T\left(
\left(  V\right)  \right)  $ for $i\neq j$ with $f_{i}\left(  h_{i}\right)
=0$ and $f_{j}\left(  h_{j}\right)  =1.$ Let $p:T\left(  \left(  V\right)
\right)  \rightarrow\mathbb{R}$ be the polynomial $p\left(  x\right)  =%
{\textstyle\prod\nolimits_{i\neq j}}
f_{i}\left(  x\right)  $ then the linear functional $L$ defined by the shuffle
product $L=f_{1}\shuffle f_{2}...\shuffle f_{n}$ agrees with $p$ on $\mathcal{S}$
and hence we arrive at the contradiction
\[
\lambda_{j}=\sum_{i=1}^{n}\lambda_{i}L\left(  h_{i}\right)  =0.
\]
\end{proof}

\begin{cor}
Let $\left\{ \gamma _{1},...,\gamma _{n}\right\} $ be a collection of
continuous $V-$valued paths of bounded variation having distinct signatures.
If $\phi :\mathbb{N\cup }\left\{ 0\right\} \rightarrow \left( 0,\infty
\right) $ satisfies Condition \ref{sum} then the matrix $K=(\left\langle
S\left( \gamma _{i}\right) ,S\left( \gamma _{j}\right) \right\rangle _{\phi
})_{i,j=1,...,n}$ is positive definite.
\end{cor}

\begin{proof}
If $0\neq x\in \mathbb{R}^{n}$ then the previous proposition ensures that $%
\sum_{i=1}^{n}x_{i}S\left( \gamma _{i}\right) _{a,b}\neq 0.$ Since $%
\left\vert \left\vert \cdot \right\vert \right\vert _{\phi }$ is a norm we
have
\[
0<\left\vert \left\vert \sum_{i=1}^{n}x_{i}S\left( \gamma _{i}\right)
_{a,b}\right\vert \right\vert _{\phi }^{2}=x^{T}Kx
\]%
as required.
\end{proof}

We now prove an existence and uniqueness theorem for the closest discrete
probability measure to Wiener measure which is supported on $\left\{
\gamma_{1},...,\gamma_{n}\right\}  .$

\begin{prop}
\label{thm-optimal-measure}Let $\left\{ \gamma _{1},...,\gamma _{n}\right\} $
be a collection of continuous $V-$valued paths of bounded variation defined
over $\left[ a,b\right] $ and having distinct signatures. Assume that  $\phi
:\mathbb{N\cup }\left\{ 0\right\} \rightarrow \left( 0,\infty \right) $
satisfies Condition \ref{sum}. Let $C_{n}$ denote the $n-$simplex $\left\{
\mu =\left( \mu _{1},...,\mu _{n}\right) :\sum_{i=1}^{n}\mu _{i}=1,\mu
_{i}\geq 0\right\} $ so that $C_{n}$ is in one-to-one correspondence with
the set of probability measures supported on $\left\{ \gamma _{1},...,\gamma
_{n}\right\} $ by the identification of $\mu $ with $\sum_{i=1}^{n}\mu
_{i}\delta _{\gamma _{i}}.$ Then there exists a unique $\mu ^{\ast }$ which
minimises $d_{\phi }\left( \mu ,\mathcal{W}\right) $ over $\mu $ in $C_{n},$%
i.e.
\[
\mu ^{\ast }=\arg \min_{\mu \in C_{n}}\left\vert \left\vert \mathbb{E}\left[
S\left( \circ B\right) _{a,b}\right] -\mathbb{E}_{X\sim \mu }\left[ S\left(
X\right) _{a,b}\right] \right\vert \right\vert _{\phi }
\]
\end{prop}

\begin{proof}
It is easy to verify that the set $C_{n}$ is a compact and convex
set in $\mathbb{R}^{n}$. Since $f(x)=\frac{1}{2}x^{T}Kx-h^{T}x$
is continuous on the compact set $C_{n}$, then $f$ is bounded and
attains its minimum on some points in the set $C_{n}$. That means
that there exist optimal solutions $x^{*}\in C_{n}$ such that
\[
f(x^{*})=\min_{x\in C_{n}}f(x).
\]
Let $m=\min_{x\in C_{n}}f(x)$ and $x_{1}^{*}$, $x_{2}^{*}\in C_{n}$
be two optimal solutions. Then, for any $\alpha\in[0,1]$, we have
\[
\alpha x_{1}^{*}+(1-\alpha)x_{2}^{*}\in C_{n}
\]
and
\[
m\leq f(\alpha x_{1}^{*}+(1-\alpha)x_{2}^{*})\leq\alpha f(x_{1}^{*})+(1-\alpha)f(x_{2}^{*})=m.
\]
Thus,
\[
\frac{1}{2}(x_{1}^{*})^{T}Kx_{2}^{*}-\frac{1}{2}h^{T}(x_{1}^{*}+x_{2}^{*})=m.
\]
Since
\[
\begin{aligned}f(x_{1}^{*})=\frac{1}{2}(x_{1}^{*})^{T}Kx_{1}^{*}-h^{T}x_{1}^{*}=m\ \text{and}\ f(x_{2}^{*})=\frac{1}{2}(x_{2}^{*})^{T}Kx_{2}^{*}-h^{T}x_{2}^{*}=m,\end{aligned}
\]
combining above three equations together, we have
\[
(x_{2}^{*}-x_{1}^{*})^{T}K(x_{2}^{*}-x_{1}^{*})=0.
\]
Since the matrix $K$ is positive definite on $\mathbb{R}^{n}$, we
must have that $x_{1}^{*}=x_{2}^{*}$. So we have concluded our proof.
\end{proof}

\begin{rem}
The next aim is to find the optimal measure in Theorem \ref%
{thm-optimal-measure} and the minised value of the objective. In some cases
this can be done explicitly. Letting \thinspace $f$ be the function in the
proof, we have the following cases:
\begin{description}
\item[Case 1] There exists $x^{\ast }\in C_{n}$ such that $\nabla f(x^{\ast
})=0$. Then the optimal solution and the value are
\[
x^{\ast }=K^{-1}h\in C_{n},\ f(x^{\ast })=-\frac{1}{2}h^{T}K^{-1}h.
\]
\item[Case 2] Assume that $\nabla f$ is non-vanishing on $C_{n}$. If there
exists a vertex $e_{m}$ of $C_{n}$ such that $f\left( e_{m}\right) <f\left(
e_{j}\right) $ for all $j\neq m$ and if it satisfies that
\begin{equation}
(Ke_{m}-h)^{T}(e_{i}-e_{m})\geq 0,\ \forall i\in \lbrack n]:=\{1,2,\cdots
,n\},  \label{eq_emim}
\end{equation}%
then the optimal solution is $e_{m}$ and $f(e_{m})=\frac{1}{2}%
e_{m}^{T}Ke_{m}-h^{T}e_{m}$. Actually, we have
\[
\begin{aligned}f(x)-f(e_{m}) & =\nabla f(e_{m})^{T}(x-e_{m})+\frac{1}{2}(x-e_{m})^{T}\nabla^{2}f(e_{m})(x-e_{m})\\
& =(Ke_{m}-h)^{T}(x-e_{m})+\frac{1}{2}(x-e_{m})^{T}K(x-e_{m})\\
& =(Ke_{m}-h)^{T}\left(\text{\ensuremath{\sum}}_{i=1}^{n}\alpha_{i}e_{i}-e_{m}\right)+\frac{1}{2}(x-e_{m})^{T}K(x-e_{m})\\
& =\text{\ensuremath{\sum}}_{i=1}^{n}\alpha_{i}(Ke_{m}-h)^{T}(e_{i}-e_{m})+\frac{1}{2}(x-e_{m})^{T}K(x-e_{m})\\
& \geq0,
\end{aligned}
\]%
where $x=\text{$\sum $}_{i=1}^{n}\alpha _{i}e_{i}$ is a convex combination
of vertexes of $C_{n}$. The condition (\ref{eq_emim}) means that
\[
\begin{aligned}\tilde{f}(t) & =f((1-t)e_{m}+te_{i})\\
& =\frac{1}{2}(e_{i}-e_{m})^{T}K(e_{i}-e_{m})t^{2}+(Ke_{m}-h)^{T}(e_{i}-e_{m})t+f(e_{m})
\end{aligned}
\]%
is increasing on the interval $[0,1]$.
\end{description}
\end{rem}

If $\nabla f$ does not vanish in $C_{n}$ and the conditions in case 2 of the
above do not hold, then there is no explicit expression for the optimal
solution and alternative numerical methods are needed to determine the
minimiser. Common tools are active-set methods and interior point methods; see \cite{Wong-2011,Wright} and the references therein). 

\section{Examples and Numerical Results\label{sec-Numerical-Application}}

In this section, we give some numerical results to illustrate the usefulness of general
signature kernels in measuring the similarity/alignment between a given
discrete measures on paths and Wiener measure. We illustrate the use of these
measures in a number of examples. As in the previous section let $\mu
=\sum_{i=1}^{n}\lambda_{i}\delta_{\gamma_{i}}$ be a discrete probability
measure supported on a finite collection of continuous bounded variation paths
$\gamma:[0,1]\rightarrow V$ and denote the Wiener measure on $\mathcal{W}$.
A plausible measure of the alignment between these two expected signatures is
\begin{equation}
\cos\angle_{\phi}(\mu,\mathcal{W}):=\frac{\left\langle \mathbb{E}_{X\sim\mu
}\left[  S\left(  \circ X\right)  _{0,1}\right]  ,\mathbb{E}_{X\sim
\mathcal{W}}\left[  S\left(  X\right)  _{0,1}\right]  \right\rangle _{\phi}%
}{\left\vert \left\vert \mathbb{E}_{X\sim\mu}\left[  S\left(  \circ X\right)
_{0,1}\right]  \right\vert \right\vert _{\phi}\left\vert \left\vert
\mathbb{E}_{X\sim\mathcal{W}}\left[  S\left(  X\right)  _{0,1}\right]
\right\vert \right\vert _{\phi}}.\label{eq-corr-mu-W}%
\end{equation}
It follows from our earlier discussion that $\cos\angle_{\phi}(\mu
,\mathcal{W})\in\left[  0,1\right]  .$ A justification for this quantity
measuring the alignment of the measures $\mu$ and $\mathcal{W}$, rather than
just their expected signatures, is that for any given pair of measures
$\nu_{1}$ and $\nu_{2}$ on a space of (rough) paths it holds that $\cos
\angle_{\phi}(\nu_{1},\nu_{2})=1$ if and only if there exists $\lambda\in%
\mathbb{R}
$ with
\[
\mathbb{E}_{X\sim\nu_{1}}\left[  S\left(  \circ X\right)  _{0,1}\right]
=\lambda\mathbb{E}_{X\sim\nu_{2}}\left[  S\left(  \circ X\right)
_{0,1}\right]  .
\]
The fact that $\lambda=1,$ and hence that the expected signatures coincide,
follows by interpreting this equality under the projection $\pi_{0}:T\left(
\left(  V\right)  \right)  \rightarrow%
\mathbb{R}
$. Another quantity we use is the MMD distance
\begin{equation}
d_{\phi}(\mu,\mathcal{W})=\left\vert \left\vert \mathbb{E}_{X\sim\mu}\left[
S\left(  \circ X\right)  _{0,1}\right]  -\mathbb{E}_{X\sim\mathcal{W}}\left[
S\left(  X\right)  _{0,1}\right]  \right\vert \right\vert _{\phi
},\label{eq-loss-mu-W}%
\end{equation}
which we have already discussed extensively.

\subsection{Discrete Measures on Brownian Paths}

In Section \ref{sec-Optimal-Discrete-Measure}, we proved the existence of a
unique optimal probability measure $\mu^{\ast}$ supported on $\{\gamma
_{1},\cdots,\gamma_{n}\}$ such that
\[
\mu^{\ast}=\arg\min_{\mu\in C_{n}}\left\vert \left\vert \mathbb{E}\left[
S\left(  \circ B\right)  _{0,1}\right]  -\mathbb{E}^{\mu}\left[  S\left(
\gamma\right)  _{0,1}\right]  \right\vert \right\vert _{\phi}^{2}.
\]
We now present an example in which $\{\gamma_{1},\cdots,\gamma_{n}\}$ is
obtained as the piecewise linear interpolation of $n$ i.i.d discretely-sampled
Brownian paths. We consider two cases for $\phi:$

\begin{enumerate}
\item  $\phi(k)=\left(  \frac{k}{2}\right)  !$ for $n\in%
\mathbb{N}
\cup\left\{  0\right\}.$We refer to the resulting $\phi$-signature kernel, somewhat inexactly, as the \textit{the factorially-weighted
signature kernel}.

\item   The original signature kernel $\phi(k)\equiv1$.
\end{enumerate}

\begin{example}
We randomly sample $n$ i.i.d. Brownian motion paths in $\mathbb{R}^{d}$. Each
path sampled over the time interval $[0,1]$, on an equally-spaced partition $0=t_{1}<t_{2}<\cdots<t_{m}=1$ with
$t_{j+1}-t_{j}=\frac{1}{m-1}$. We denote the resulting finite set piecewise linearly interpolated Brownian sample paths as
\[
\mathcal{S}(n,m,d)=\{B_{i}\}_{i=1}^{n}\ \text{with}\ B_{i}=\{B_{i}(t_{j}%
)\in\mathbb{R}^{d}\}_{j=1}^{m}.
\]
Figure \ref{fig_BM} and \ref{fig_BM-original}, displays the
alignment $\cos\angle_{\phi}(\mu^{*},\mathcal{W})$ and the similarity $d_{\phi}(\mu^{*},\mathcal{W})$
for the optimal discrete probability measure supported on
$\mathcal{S}(n,m,d)$, in which the number
of sample paths $n=10$ and the observation points $m=10$ are fixed and the
dimension $d$ is varied over the range 2 to 6. We run 400 independent experiments
for each $d$, that is, we generate 400 independent samples of the sets $\mathcal{S}(n,m,d)$ for each
dimension $d$. Each set $\mathcal{S}(n,m,d)$ has an optimal measure associated with it, which we compute. The boxplots in Figure
\ref{fig_BM} and \ref{fig_BM-original} show the median, range and interquartile range of the values of the alignment and
the similarity of the optimal discrete measures over these 400 samples.
Qualitatively we can see from both quantities show dependence on the dimension of the state space, with the
alignment decreasing and the dis-similarity increasing w.r.t. the dimension. We can also compare
the results using the two different $\phi$-signature kernels with the original signature kernel showing the same behaviour
w.r.t. the dimension having a persistently higher level of alignment than under the factorially-weighted signature kernel
across all of the dimensions considered.
\end{example}

\begin{figure}[tbh]
\begin{centering}
\includegraphics[scale=0.55]{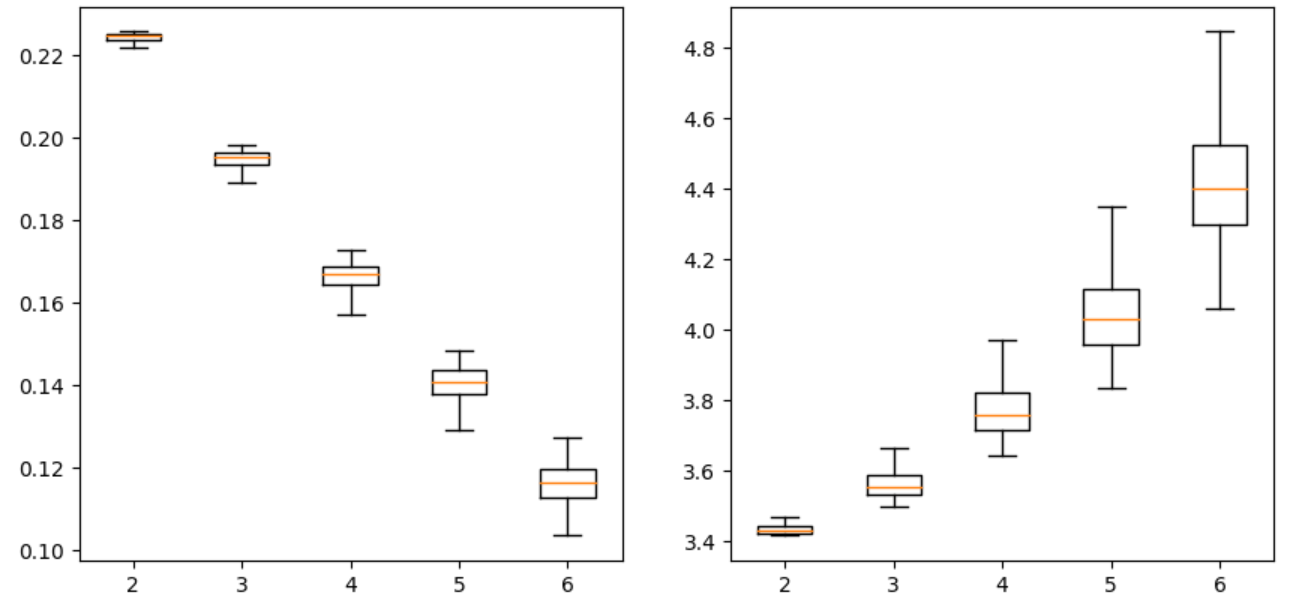}
\par\end{centering}
\caption{Boxplots of the factorially-weighted signature kernel.
(a) The left panel shows the distribution of the values of the alignment $\cos\angle_{\phi}(\mu^{*},\mathcal{W})$ of the optimal measure and the Wiener
measure across 400 samples. The x-axis is the dimension of the Brownian motion, and the y-axis the value of the
alignment. (b) The right panel shows the same for the MMD distance $d_{\phi}(\mu^{*},\mathcal{W}).$}
\label{fig_BM}
\end{figure}

\begin{figure}[tbh]
\begin{centering}
\includegraphics[scale=0.43]{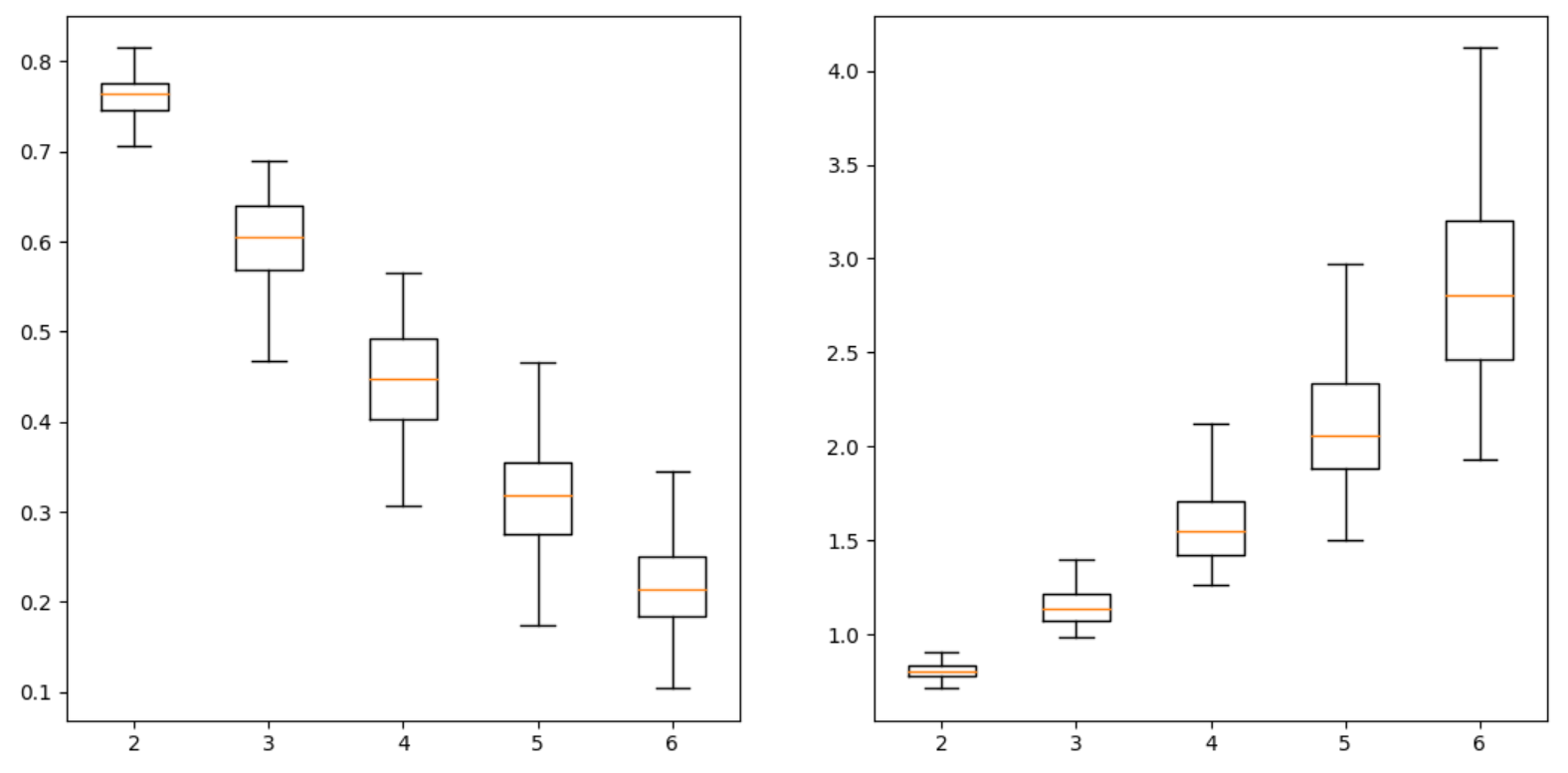}
\par\end{centering}
\caption{The optimal measure under the original signature kernel}%
\label{fig_BM-original}%
\end{figure}

\subsection{Examples using cubature formulae}

In the paper \cite{LV-2004}, Lyons and Victoir studied cubature on
Wiener space. Let $C_{bv}([0,T],V)$ be a subset of Wiener space made of
bounded variation paths. We say that the paths $\gamma_{1},\cdots,\gamma
_{n}\in C_{bv}([0,T],V)$ and the positive weights $\lambda_{1},\cdots
,\lambda_{n} $ define a cubature formula on Wiener space of degree $m$ at time
$T$ if
\[
\mathbb{E}\left[  S\left(  \circ B\right)  _{0,T}(e_{I}^{*})\right]
=\sum_{j=1}^{n}\lambda_{j}S\left(  \gamma_{j}\right)  _{0,T}(e_{I}^{*})
\]
for all $I\in\mathcal{A}_{m}:=\{I=(i_{1},\cdots,i_{k}):k\leq m\}$ with
$m\in\mathbb{N}$.

Cubature on Wiener space can be an effective way to develop high-order numerical schemes for
high-dimensional stochastic differential equations and parabolic
partial differential equations, see \cite{LV-2004}. In Section 5 of
\cite{LV-2004}, the authors also construct an explicit cubature formula of degree 5 for 2-dimensional Brownian motion. The reader
can find formulas of these cubature paths and measure in tables 2 and 3 in the same reference.

In this subsection, we analyse the results for a family of $\phi$-signature
kernels on three discrete probability measures supported on this collection of cubature paths. We consider the cubature weights themselves, the empirical measure of the sample (i.e. where they are equally weighted) and the optimal measure obtained from 
Section \ref{sec-Optimal-Discrete-Measure}. In Figure \ref{fig_cubature_Beta},
we show the similarity of these discrete measures and the Wiener measure under
the family of Beta-weighted signature kernels given by  
\begin{equation}
\phi(k)=\frac{\Gamma(m+1)\Gamma(k+1)}{\Gamma(k+m+1)}
\end{equation}
for various values of $m$ in the weight $\phi$ (shown along the horizontal axis).

The plot on the left panel of \ref{fig_cubature_Beta} shows that as the parameter $m$ increases these three
distances first increase fast and then gradually go down. We see that the distance of the optimal measure and the Wiener measure is
smallest and the distance of the empirical measure is much larger than the
distance of cubature measure.  The right
pane shows the ratio of the distance of optimal measure and the distance
of cubature measure for different choices of $m$.

\begin{figure}[tbh]
\begin{centering}
\includegraphics[scale=0.45]{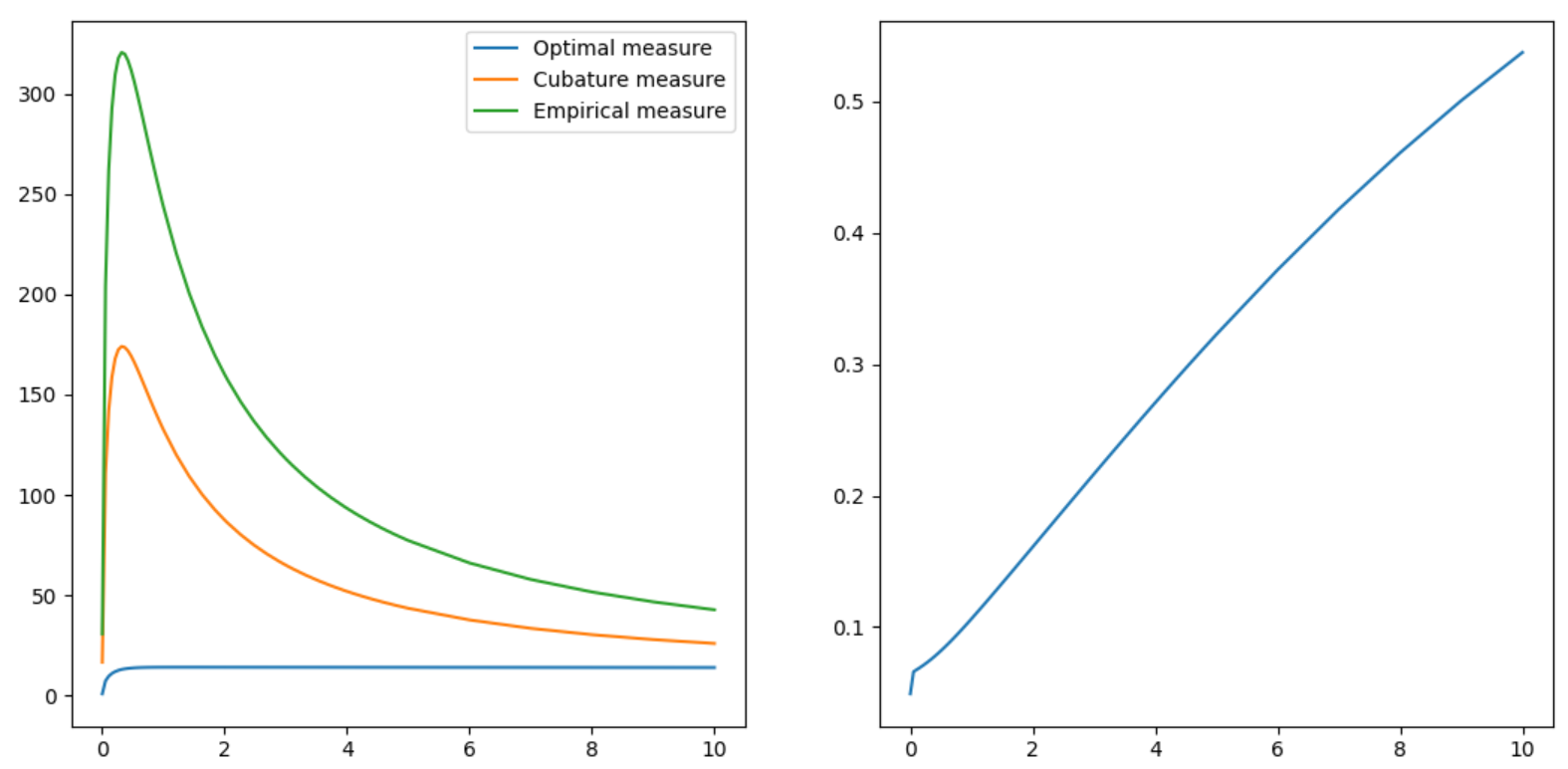}
\par\end{centering}
\caption{The similarity under a family of Beta-weighted signature kernels. The
left panel is the plot of the distance of these discrete measures and the
Wiener measure plotted against different values of $m$ on the horizontal axis. The right
panel plots the ratio of the optimal distance and the cubature
distance.}%
\label{fig_cubature_Beta}%
\end{figure}

\subsection{Applications in Signal Processing}

The alignment in equation (\ref{eq-corr-mu-W}) and the similarity in equation
(\ref{eq-loss-mu-W}) defined by the $\phi$-signature kernel give us a way of
determining how large a given discrete measure is different to the Wiener
measure. We can use these quantities to measure deviation of a discrete
measure from a reference measure (i.e. the Wiener measure here). A natural
application of these methods in signal processing is to mitigate/detect the (additive) contamination of 
white noise under different types of perturbation. 

The examples studied here are motivated by an attempt to study radio frequency interference
(RFI) in the radio astronomy. In this setting astronomers would like to obtain
high-resolution sky images of an interested astrophysical object using measurements from an array of antennas 
(e.g. the Karl G. Jansky Very Large Array
(VLA) etc.). To observe the sky and then synthesis the sky image interested.
The observation is called visibility $V_{ij}(t,v,p)$, where $ij$ is an antenna
pair, $t$ is the time integration, $v$ is the frequency and $p$ is the
polarization. Usually the visibility would be contaminated by thermal noise
and radio frequency interference (RFI). So the observation data from an
interferometer can be broken down into three components: the astrophysical sky
signals, thermal noise and RFI. The first component is slowly varying which
can be removed in the observation data by sky-subtraction method (see e.g.
\cite{Wilensky2019}). The RFI signal is usually much stronger than thermal
noise but is also sometimes ultra-faint. For different antennas, the RFI contamination
is systematic and thermal noise can be assumed to be independent. In order to
obtain a high-resolution image, the first step
is to design some methods to identify and then, if possible, to remove the RFI component of the observation. 

We consider two idealised types of RFI contamination. The first is by simple superposition with a 
sine wave of a fixed single frequency and a given amplitude and phase, so that the interference is narrow-band but persistent over time. The second will be to consider a short duration spike, as modelled in the paper by Davis and Monroe \cite{Davis-1984} in the univariate setting, in which the Brownian signal undergoes a perturbation at a uniformly distributed random time to give
\begin{equation}
B(t)+\epsilon\sqrt{(t-U)^{+}}. 
\end{equation}

We again compare the use of two $\phi$-signature kernels. The
factorially-weighted signature kernel and the original
signature kernel.

\begin{example}
Working in $d$-dimensions we take a path of the form
\[
X_{i}^{(j)}(t)=B_{i}^{(j)}(t)+\epsilon\sin(2\pi\nu t-\phi_{i}^{(j)}%
),\ j=1,2,\cdots,d
\]
where the frequency $\nu$ is fixed, the phase shifts are $\phi_{i}^{(j)}$ and $\epsilon$ denotes a (small) fixed amplitude. Let a finite
collection of sample paths on time interval $[0,1]$ as
\[
\mathcal{S}(n,m,d)=\{X_{i}\}_{i=1}^{n},\ \text{where}\ X_{i}=\{(X_{i}%
^{(1)}(t_{j}),\cdots,X_{i}^{(d)}(t_{j}))\in\mathbb{R}^{d}\}_{j=1}^{m}.
\]
In Figure \ref{fig_BMsine} and \ref{fig_BMsine-original}, we fixed
$(n,m,d)=(10,10,2)$, $\epsilon$ from $[0,1]$and the frequency $\nu\in\{2,3\}$.
We run 100 collections of paths $\mathcal{S}(n,m,d)$ for each $\epsilon$ and
frequency $\nu$. The figures show the deviation of the alignment and the
similarity of the optimal measure (the empirical measure, resp.) and the
Wiener measure, in which the middle line is the median of the alignment or the
similarity resp., and the shadow represents the range from the lower
quartile to the upper quartile. We generate 100 experiments for each
$\epsilon$.
The figures show that the alignment decreases very fast to a low level and
the dis-similarity increases very quickly as $\epsilon$ becomes large for both the
optimal measure and the empirical measure. At larger frequencies $\nu$, the alignment (dis-similarity) decays (grows) more rapidly. 
\end{example}

Finally we present an example based on the construction in the paper of Davis and Monroe \cite{Davis-1984} mentioned earlier. 
Here the interference is characterised by a sudden high energy spike at a uniform random time. 

\begin{figure}[tbh]
\begin{centering}
\includegraphics[scale=0.55]{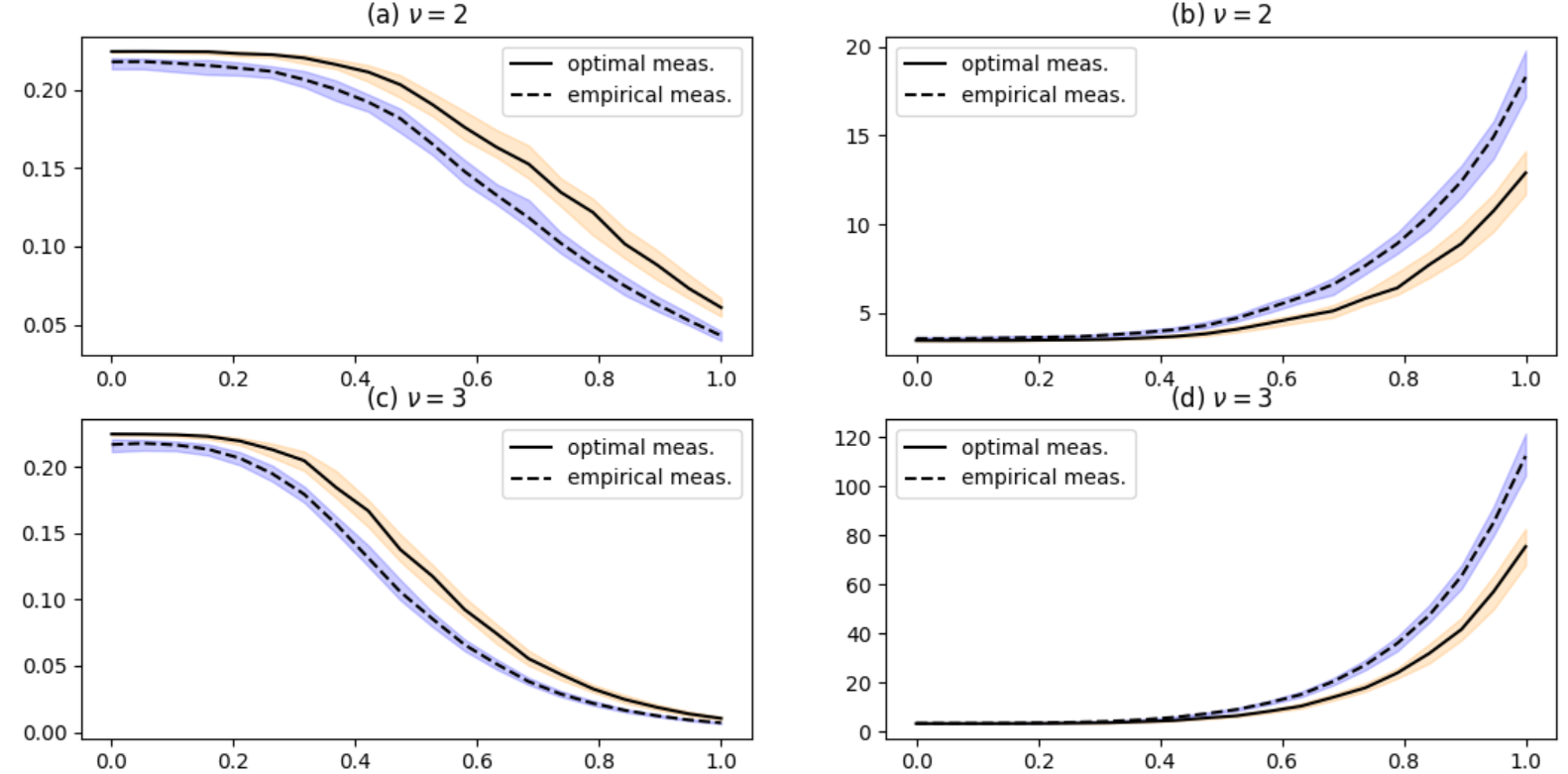}
\par\end{centering}
\caption{The case for the factorially-weighted signature kernel. (a) and (c)
show similarities of discrete measures and the Wiener measure where the horizontal is
the value of $\epsilon$ and vertical axis is the value of alignment. (b) and (d) show
similarities of discrete measures and the Wiener measure. The solid line is
for the optimal measure while the dashed line is for the empirical measure. The
upper panel is for the frequency $\nu=2$ and the lower is for $\nu=3$.}%
\label{fig_BMsine}%
\end{figure}

\begin{figure}[tbh]
\begin{centering}
\includegraphics[scale=0.55]{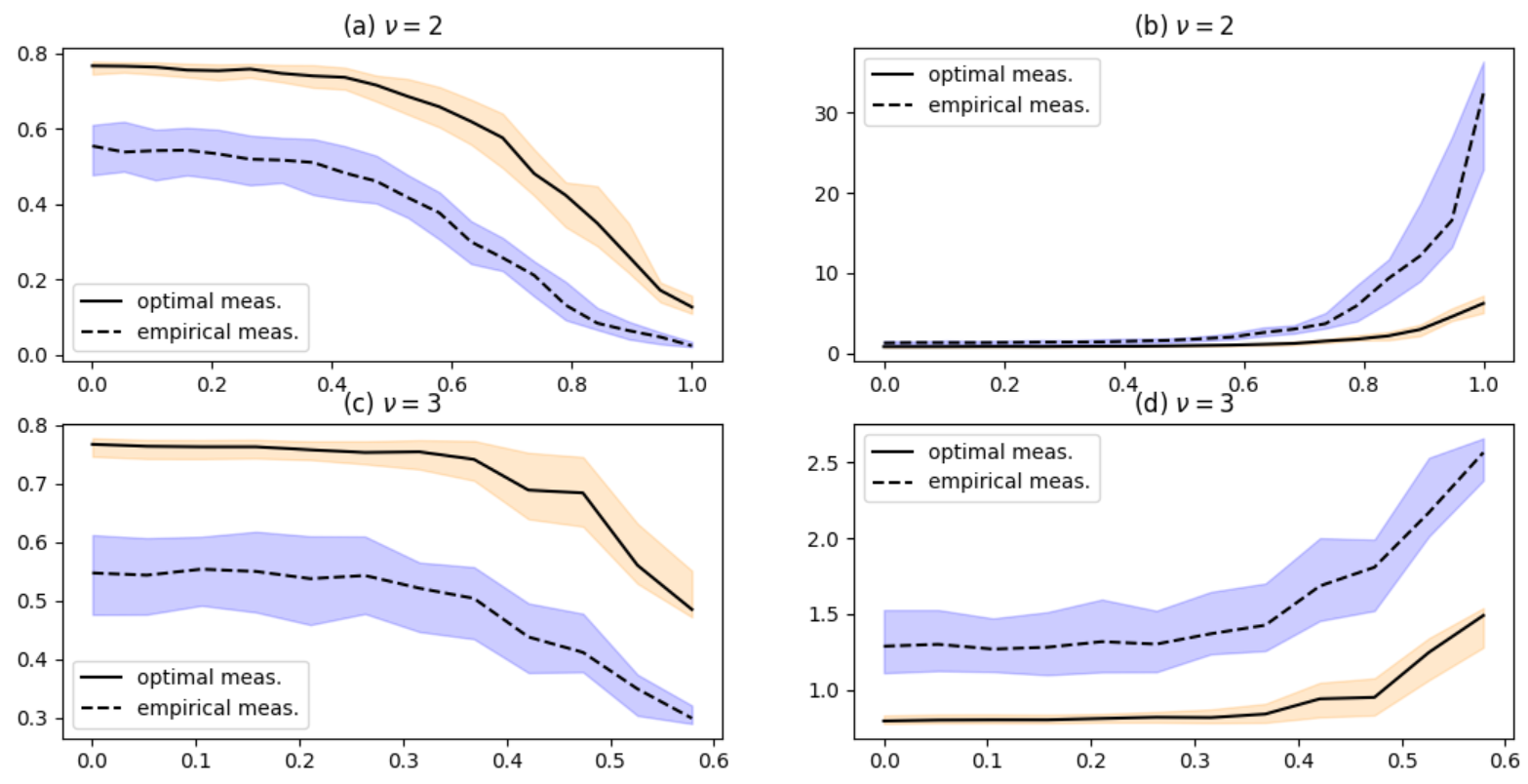}
\par\end{centering}
\caption{The same example under the original signature kernel}%
\label{fig_BMsine-original}%
\end{figure}

\begin{example}
We define
\[
X_{i}^{(j)}(t)=B_{i}^{(j)}(t)+\epsilon\sqrt{(t-U_{i})^{+}},\ j=1,2,\cdots,d
\]
where $U_{i}$ is uniformly distributed in $[0,1]$, the time interval
$t\in[0,1]$ and $x^{+}=\max\{0,x\}$. We denote a finite collection of these
paths as
\[
\mathcal{S}(n,m,d)=\{X_{i}\}_{i=1}^{n},\ \text{where}\ X_{i}=\{(X_{i}%
^{(1)}(t_{j}),\cdots,X_{i}^{(d)}(t_{j}))\in\mathbb{R}^{d}\}_{j=1}^{m}.
\]
In Figure \ref{fig_BM_Davis} and \ref{fig_BM_Davis-original}, the parameters
$(n,m,d)=(10,10,2)$ are fixed and $\epsilon$ is taken from $[0,5]$. We run 100
independent experiments for each $\epsilon$. The plots are like ones in the
above example. The middle line is the median of the alignment (the similarity,
resp.) and the shadow is the range from the lower quartile to the upper
quartile of the alignment (the similarity, resp.) for the 100 collections of
sample paths.
We can see from these figures that the alignment (the dis-similarity, resp.) is
decreasing (increasing, resp.) as $\epsilon$ increases, as one would expect. From the point
view of RFI mitigation, the alignment of the empirical measure is more
relevant than that of the optimal measure. It is reasonable that as the
strength $\epsilon$ is large the empirical measure is less similar w.r.t. the
Wiener measure than the optimal measure. The alignment of the empirical measure decays faster than that of optimal measure
in our experiments. This suggests potential uses for building method for the identification of RFI based on a threshold for the alignment of the empirical measure. The preliminary results here for instance suggest that a threshold of alignment of 0.2
under the factorially-weighted signature kernel could be used in this example.
\end{example}

\begin{figure}[tbh]
\begin{centering}
\includegraphics[scale=0.45]{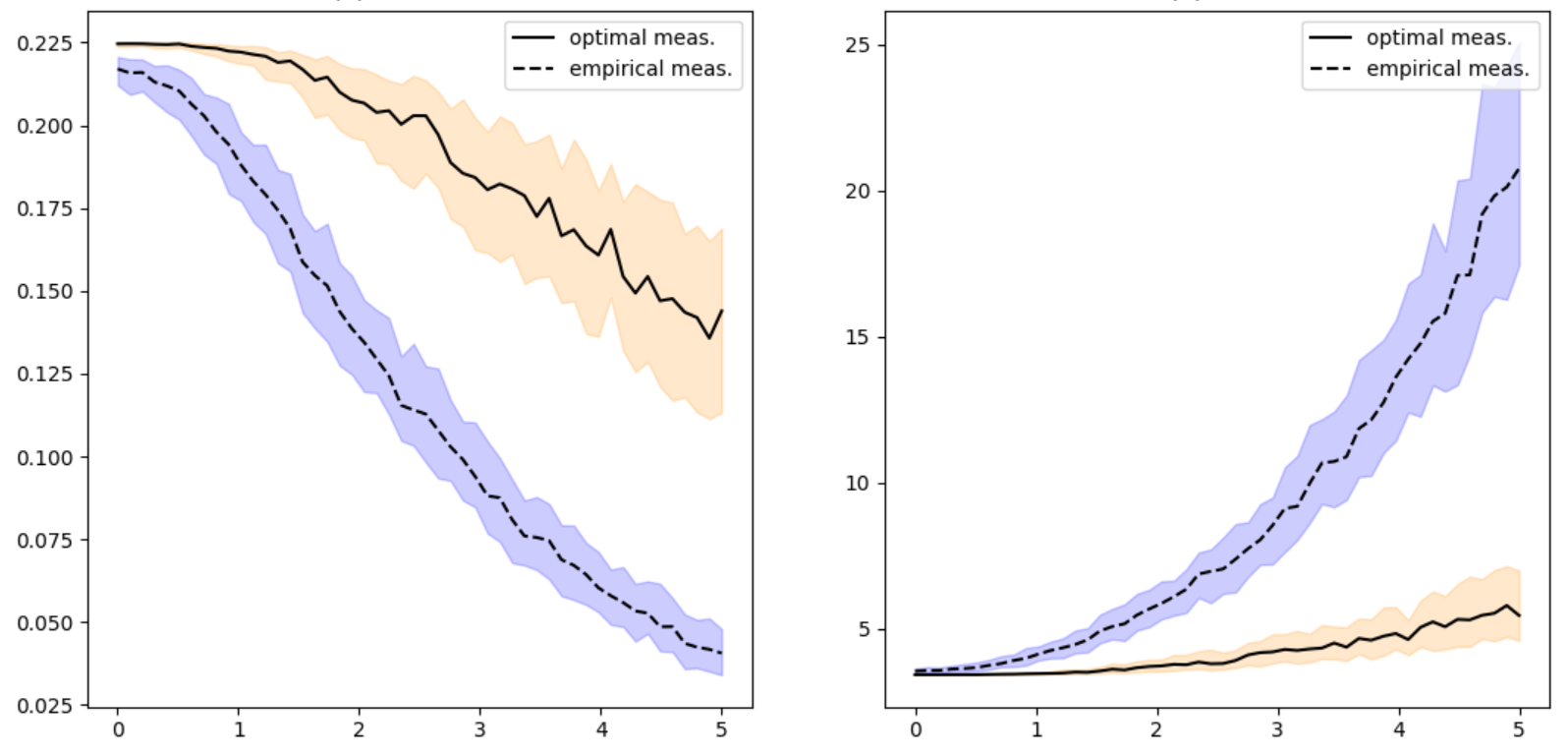}
\par\end{centering}
\caption{The case for the factorially-weighted signature kernel. (a) The left
panel shows the alignment of discrete measures and the Wiener measure for each
$\epsilon$ taken from $[0,1]$ where x-axis is the value of $\epsilon$ and
y-axis is the value of alignment. (b) The right panel shows the similarity of
discrete measures and the Wiener measure as in (a). The solid line is for the
optimal measure while the dash line is for the empirical measure.}%
\label{fig_BM_Davis}%
\end{figure}

\begin{figure}[tbh]
\begin{centering}
\includegraphics[scale=0.45]{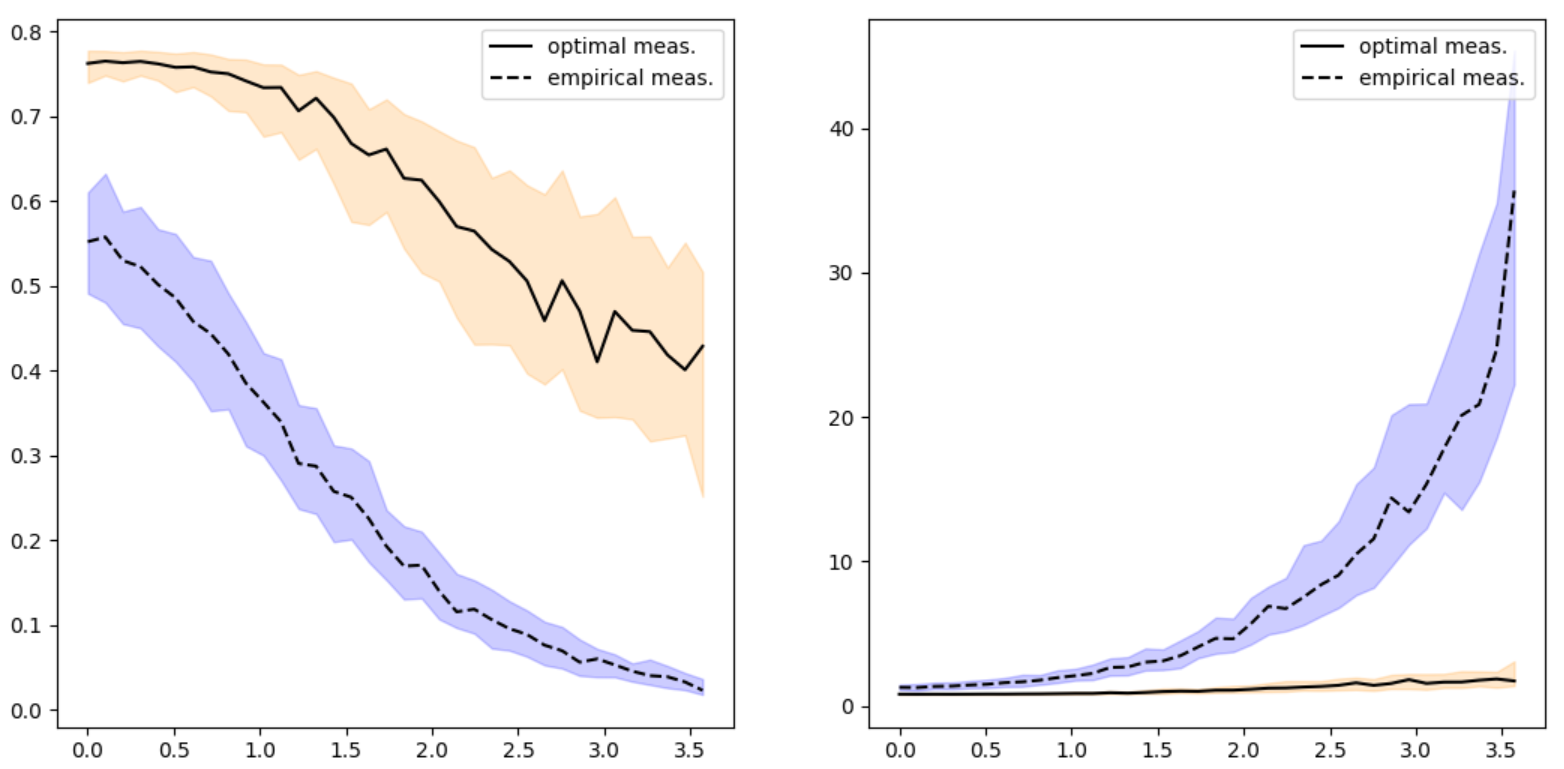}
\par\end{centering}
\caption{The same example under the original signature kernel}%
\label{fig_BM_Davis-original}%
\end{figure}

\newpage{}

\clearpage

\section*{Acknowledgements} The authors thank Cris Salvi and Bojan Nikolic respectively for discussions related to the Signature Kernel and the problem of RFI mitigation in radioastronomy.

\end{document}